\documentclass{amsart}
\usepackage{amsmath,amssymb, amsthm, amscd,mathrsfs}

\def\Af{\mathcal{A}}
\def\mto{\mathbb{L}}

\def\H{H}
\def\real{\mathbb{R}}
\def\complex{\mathbb{C}}
\def\integer{\mathbb{Z}}

\def\FBI{\mathcal{T}}
\def\pFBI{\mathscr{T}}
\def\Sch{\mathcal{S}}
\def\Pj{\mathcal{P}}
\def\pPj{\mathscr{P}}
\def\Fourier{\mathscr{F}}
\def\L{\mathcal{L}}
\def\Id{\mathrm{Id}}

\def\weight{\mathcal{W}}
\def\cone{\mathbf{C}}
\def\ctr{\mathrm{ctr}}
\def\cpt{\mathrm{cpt}}
\def\hyp{\mathrm{hyp}}
\def\aniso{\mathrm{aniso}}
\def\supp{\mathrm{supp}}
\def\k{\mathbf{k}}
\def\LL{\mathbb{L}}

\newtheorem{theorem}{Theorem}[section]
\newtheorem{proposition}[theorem]{Proposition}
\newtheorem{lemma}[theorem]{Lemma}

\newtheorem{corollary}[theorem]{Corollary}
\newtheorem*{claim}{Claim}

\theoremstyle{remark}
\newtheorem{remark}[theorem]{Remark}

\theoremstyle{definition}
\newtheorem{definition}[theorem]{Definition}

\begin{document}

\title{Contact Anosov flows and the FBI transform}
\author{Masato TSUJII}
\address{Department of Mathematics, Kyushu University, 
Moto-oka 744, Nishi-ku, Fukuoka, 819-0395, JAPAN}
\email{tsujii@math.kyushu-u.ac.jp}
\date{}
\begin{abstract}
This paper is about spectral properties of transfer operators for contact Anosov flows. 
The main result gives the essential spectral radii of the transfer operators  acting on an appropriate function space exactly and  improves the previous result in~\cite{Tsujii2009}. 
Also we  provide a simplified proof by using the so-called FBI (or Bargmann) transform. 
\end{abstract}

\maketitle

\section{Introduction}
In this paper, we consider spectral properties of transfer operators for contact Anosov flows.
A contact Anosov flow is by definition an Anosov flow preserving a contact form  on the underlying manifold.
The geodesic flows on closed negatively curved manifolds are  typical examples of  contact Anosov flows and have been studied extensively since the work\cite{Hopf39}
of E.\ Hopf in 1930's.  We refer  \cite{Anosov69, Sinai61,Dolgopyat98, Liverani04, Tsujii2009} and the references therein for succeeding works related to this paper. 

For a contact Anosov flow $F^t:M\to M$ and a multiplicative cocycle $g^t$ over it, we consider the one-parameter family of transfer operators
\[
\L^t u=g^t\cdot u\circ F^t.
\]
In the previous paper\cite{Tsujii2009}, we studied the case $g^t\equiv 1$, that is, the case of pull-back operator and proved that the operators $\L^t$ for sufficiently large $t$ are quasi-compact if we choose an appropriate function space for them to act on. This implied not only exponential decay of correlations, which had been proved by Dolgopyat\cite{Dolgopyat98} and Liverani\cite{Liverani04}, but also a precise asymptotic formula for correlation decay or the Ruelle-Pollicott resonance. 
The main result of the present paper generalizes it to transfer operators with general cocycles $g^t$ and improves the statement slightly  by giving the essential spectral radius of $\L^t$  exactly in terms of dynamical exponents.
Besides we provide a simplified proof for the main  result by introducing a technique, the FBI transform, from semi-classical analysis. 

The basic idea behind our argument is to regard functions on the  manifold $M$ as superpositions of wave packets (that is, localized simple wave functions) and to study  the action of transfer operator $\L^t$ on each of those wave packets. 
The wave packets are parametrized by elements of the cotangent bundle $T^*M$, that is,  pairs of a point $x\in M$, which indicates the center of mass, and a cotangent vector $\xi\in T_x^*M$, which indicates the frequency vector. The action of the transfer operator $\L^t$ of the wave packets are closely related to the action $(DF^t)^*:T^*M\to T^*M$ of the flow $F^t$ on the cotangent bundle $T^*M$. That is, roughly speaking, the wave packet corresponding to $(x,\xi)\in T^*M$ is transferred by $\L^t$ to a superposition of wave packets corresponding to elements of $T^*M$ close to  $(DF^t)^*(x,\xi)$. 
Notice that the action of the flow $F^t$ on $T^*M$ is not recurrent outside any small neighborhood of the one-dimensional subbundle spanned by $\alpha$ and also that the part of $\L^t$ acting on wave packets with low frequency should be compact. Therefore, to estimate the essential spectral radius of $\L^t$,  we mainly concern the action of $\L^t$ on wave packets that have high frequency in the direction of the contact form $\alpha$.
That is, in the case of geodesic flows, we concern the situation where wave packets with high frequency are proceeding along the geodesic curves in the corresponding directions. This reminiscent  us of the situation studied in  semiclassical analysis and suggests a vague idea that some argument and technique in semi-classical analysis may be useful in the study of  "classical" geodesic flows (or contact Anosov flows, more generally). In this paper, we would like to show that this is the case in fact. 
See also the recent paper \cite{Faure10} for a similar approach, in which Faure and Sj\"ostrand use semi-classical analysis to give an asymptotic estimate for the eigenvalues of the generator of transfer operators.

\section{The main result}
Let $M$ be a closed odd-dimensional $C^\infty$ Riemann manifold of dimension $2d+1$.
We suppose that $M$ is equipped with a $C^\infty$ contact form $\alpha$, which is by definition a differential $1$-form such that  $\alpha\wedge (d\alpha)^d$  vanishes nowhere on $M$. 
A $C^\infty$ flow $F^t:M\to M$ is said to be Anosov if there exist constants $C>1$, $\lambda_0>1$ and a continuous invariant decomposition $TM=E_c\oplus E_s\oplus E_u$ of the tangent bundle such that 
\begin{enumerate}
\item $E_c$ is the $1$-dim subbundle spanned by the generating vector field~$V$ of $F^t$.
\item $\|DF^t(v)\|\le C \lambda_0^{-t}\|v\|$ for any $v\in E_s$ and $t\ge 0$.
\item $\|DF^{-t}(v)\|\le C \lambda_0^{-t}\|v\|$ for any $v\in E_u$ and $t\ge 0$.
\end{enumerate}
If an Anosov flow $F^t:M\to M$ preserves the contact form $\alpha$, we call it a contact Anosov flow.
For a contact Anosov flow, the subspaces $E_s$ and $E_u$ is contained in the kernel of $\alpha$ because of invariance of $\alpha$ and, hence, $\ker \alpha=E_s\oplus E_u$. 
The restriction of $d\alpha$ to $\ker \alpha$ is a symplectic form from the definition of contact form and vanishes on $E^u$ and $E^s$ because it is invariant with respect to the flow $F^t$. 
This implies $\dim E_s=\dim E_u=d$ in particular.

A $C^\infty$ one-parameter family of functions $g^t:M\to \complex\setminus\{0\}$ with parameter $t\in \real$ is called a multiplicative cocycle over a flow $F^t:M\to M$ if 
$g^{t+s}(x)=g^t(F^s(x))\cdot g^s(x)$ for $t,s\in \real$ and $x\in M$.
For a flow $F^t$ and a multiplicative cocycle $g^t$ over it, we consider a transfer operator
\[
\L^t:C^\infty(M)\to C^\infty(M), \qquad \L^t \, u(x)=g^t(x)\cdot u(F^t(x)).
\]
This is a one-parameter group of operators. In what follows, we consider the transfer operator $\L^t$ associated to a $C^\infty$ contact Anosov flow $F^t:M\to M$ and a $C^\infty$ multiplicative cocycle $g^t:M\to \complex$. 
In Section \ref{s:pFBI}, we will introduce a scale of Hilbert spaces $\H_{\aniso}^r(M)$ with a parameter $r>0$, adapted to the flow $F^t$. Those Hilbert  spaces satisfy 
\[
H^{r}(M)\subset H_\aniso^r(M) \subset H^{-r}(M)
\]
where $H^r(M)$ is the Sobolev space  of order $r$ on $M$. Our main result  is 
\begin{theorem}\label{th:main} 
If $t$ is sufficiently large, the transfer operator $\L^t$ extends naturally to a bounded operator on $H_\aniso^r(M)$ for any $r>0$. If $r>0$ is sufficiently large, the essential spectral radius of $
\L^t:H_\aniso^r(M)\to H_\aniso^r(M)
$ is exactly $\Lambda^t$,  where $\Lambda$ is the quantity defined by
\[
\Lambda=\lim_{t\to \infty}\sup_{x\in M} \left(\frac{|g^t(x)|}{\sqrt{\det (DF^t|_{E_u})}}\right)^{1/t}.
\]
\end{theorem}

The conclusion of the theorem implies in particular that, if $t$ and $r$ are sufficiently large, the spectral set of the operator $\L^t:H_\aniso^r(M)\to H_{\aniso}^r(M)$ on the outside of the disk $|z|\le \Lambda^t$ consists of discrete eigenvalues with finite multiplicities. (See \cite{Tsujii2009} for implication of this conclusion on decay of correlations.)  

Note that, if $F^t$ is the geodesic flow of a closed surface with constant negative ($\equiv -1$) curvature and $g^t\equiv 1$, we have $\Lambda=e^{-1/2}$ and the bound $\Lambda^t=e^{-t/2}$ on the essential spectral radius of $\L^t$ in the theorem above  is exactly the optimal one expected from the classical result of Selberg\cite{McKean72} and the heuristic relation between the zeros of dynamical zeta functions and the  spectrum of the generator of $\L^t$. (See also \cite{Moore,Ratner}.)
We expect that the bound in the theorem above is optimal as far as we consider the action of $\L^t$ on Banach spaces $B$ such that $C^\infty(M)\subset B\subset (C^\infty(M))'$. 
Note that the peripheral eigenvalues outside the essential spectral radius is essentially independent of the choice of function spaces. (See \cite[Appendix A]{BaladiTsujii08}.)

In the following sections, we proceed as follows. 
In Section \ref{s:redloc}, we set up a Darboux coordinate system on $M$ and reduce the main theorem to the corresponding claim (Theorem \ref{th:localmain}) about transfer operators on the Euclidean space $\real^{2d+1}$ equipped with a standard contact form $\alpha_0$.  
In Section \ref{s:FBI}, we discuss about the FBI transform on the Euclidean space $\real^{2d}$. 
The FBI transform decomposes functions on $\real^{2d}$ into Gaussian wave packets parametrized by points in the  cotangent bundle $T^*\real^{2d}=\real^{2d}\oplus \real^{2d}$ and was used in semi-classical analysis, by Sj{\"o}strand\cite{Sjostrand} and Martinez\cite{Martinez}, in order to study microlocal properties of functions.
In Section \ref{s:pFBI}, we introduce the partial FBI transform on the Euclidean space $\real^{2d+1}$, which is roughly a combination of the Fourier transform in the direction of the flow and the (scaled) FBI transform in the transversal directions. We then define the anisotropic Sobolev space $H^r_{\aniso}$ by using the partial FBI transform and a weight function $\weight^r_\aniso$.
In the last subsection, we also study the action of a linear transformation with some hyperbolic property on the anisotropic Sobolev space $H^r_{\aniso}$ and prove an analogue of 
Theorem \ref{th:localmain} in this simple case. 
In Section \ref{s:lt}, we will decompose the transfer operator on the Euclidean space $\real^{2d+1}$ into three parts, {\em the compact, hyperbolic and central part}. 
The compact part concerns the wave packets with low frequency and is  a compact operator as its name indicates. 
The hyperbolic part concerns the wave packets that have high frequency in the transversal directions to the flow.  In Section \ref{s:hyp}, we will show that the operator norm of the hyperbolic part is bounded by an arbitrarily small constant if we take a large parameter $r$ in the definition of $H^r_{\aniso}$. Thus the central part turns out to be most essential for our argument. We study the central part in Section~\ref{s:ctr}. We will decompose the central part into small pieces so that each piece can be well approximated by the simple operator studied in the last subsection of Section~\ref{s:pFBI}. 
In Section~\ref{sec:lb}, we give a lower bound for the essential spectral radius that coincides with the upper bound, finishing the proof of the main theorem.

\medskip
\noindent{\bf Acknowledgement.} 
The author expresses his gratitude to the Mittag-Leffler institute and  Prof.\ M.\ Benedicks (KTH) for hospitality during his stay at the institute for the program "Dynamics and PDE", where he started to write a preliminary version of this paper. 
The author also thank F.~Faure (Fourier Institute) for teaching him about basic facts about the FBI transform and the ingenious change of coordinates (see Subsection \ref{ss:sep}) that he introduced  in the paper\cite{Faure07}.

\section{Local properties of contact Anosov flow}\label{s:redloc}
\subsection{Darboux local coordinates}
Let $F^t:M\to M$ be a contact Anosov flow and $\alpha$ the contact form on $M$ preserved by the flow. By multiplying $\alpha$ by some smooth function, we may and do suppose that $\alpha(V)\equiv 1$ for the generating vector field $V$ of the flow $F^t$. 

On the $2d+1$ dimensional Euclidean space $\real^{2d+1}$, the standard contact form $\alpha_0$
 is defined by 
\begin{align}
\alpha_0&= dx_0 +\left(\sum_{j=1}^{d} x_j\cdot dx_{d+j}-x_{d+j}\cdot dx_j\right).
\end{align}
Note that the vector field $\partial_{x_0}:=\partial/\partial x_0$ is characterized by the conditions
\begin{equation}\label{Reeb}
\alpha_0(\partial_{x_0})\equiv 1, \qquad d\alpha_0(\partial_{x_0}, \cdot)\equiv 0.
\end{equation}
For a real number $\theta>0$, we consider the cones 
\begin{align*}
\cone_+(\theta)&=\{x=(x_0, x^+,x^-)\in \real^{2d+1}\mid \|x^-\|\le \theta\|x^+\|\},\\
\cone_-(\theta)&=\{x=(x_0, x^+, x^-)\in \real^{2d+1}\mid \|x^+\|\le \theta\|x^-\|\}
\end{align*}
where $\|\cdot \|$ is the Euclidean norm and, for $x=(x_i)_{i=0}^{2d}\in \real^{2d+1}$, we set
\[
x^+=(x_1,x_2, \cdots, x_d), \qquad  x^-=(x_{d+1}, x_{d+2}, \cdots , x_{2d}).
\]
\begin{definition}
For $\lambda>1$, a 
$C^\infty$ diffeomorphism $F:U\to U'=F(U)$ is said to be a $\lambda$-hyperbolic contact diffeomorphism if it satisfies the following conditions:
\begin{itemize}
\setlength{\itemsep}{6pt plus 1pt minus 1pt}
\item[(H1)] $U$ and $U'$ are open subsets contained in the unit disk in $\real^{2d+1}$, 
\item[(H2)] $F$ preserves the standard contact form $\alpha_0$, and 
\item[(H3)] $F$  is hyperbolic in the sense that 
\begin{align*}
&DF(\real^{2d+1}\setminus \cone_-(1/10))\subset \cone_+(1/10),\\
&DF^{-1}(\real^{2d+1}\setminus  \cone_+(1/10))\subset \cone_-(1/10)
\end{align*}
and that
\begin{align*}
&\|DF(\pi_\dag(v))\|\ge \lambda \|\pi_\dag(v)\|\quad\mbox{ for $v\in \real^{2d+1}\setminus  \cone_-(1/10)$, }\\
&\|DF^{-1}(\pi_\dag(v))\|\ge \lambda \|\pi_\dag(v)\|\quad \mbox{ for $v\in \real^{2d+1}\setminus  \cone_+(1/10)$}
\end{align*}
where $\pi_\dag:\real^{2d+1}\to \{0\}\oplus \real^{2d}$ denotes the orthogonal projection to the components other than the first one.
\end{itemize}
\end{definition}

From Darboux theorem for contact structure\cite[pp.168]{Aebischer}, it follows
\begin{lemma}[{\cite[Proposition~2.2]{Tsujii2009}}]\label{lm:Darboux}
There exist a finite system of coordinate charts 
\[
\left\{\kappa_i:V_i\to U_i=\kappa(V_i)\subset \real^{2d+1}\right\}_{i=1}^{\ell}
\]
on $M$ and a constant $c>0$ such that, for $1\le i, j\le \ell$, 
\begin{enumerate}
\item $\alpha=\kappa_i^*(\alpha_0)$ on $V_i$\quad  and \quad $\partial_{x_0}=(\kappa_i)_*(V)$ on $U_i$.
\item For sufficiently large $t>0$, the diffeomorphism induced on the local charts
\[
F_{ij}^t:=\kappa_j\circ F^t\circ \kappa_i^{-1}:U_{ij}^t\to F^t_{ij}(U^t_{ij})\quad \mbox{where }\; U_{ij}^t:=\kappa_i(V_{i}\cap F^{-t}(V_j))
\]
 is a $c \cdot \lambda_0^t$-hyperbolic contact diffeomorphism, provided $U_{ij}^t\neq\emptyset$,
\end{enumerate}
where $\lambda_0>1$ is the constant in the definition of Anosov flow. 
\end{lemma}

We henceforth fix a system of coordinate charts as above and define a family  of transfer operators on the local charts as follows. 
Take functions $\rho_i\in C_0^\infty(U_i)$, $1\le i\le \ell$, so that the family $\{\rho_i\circ \kappa_i\}$ of functions on $M$ is a $C^\infty$ partition of unity subordinate to the open covering $\{V_i\}$ and take another family of functions $\tilde{\rho}_i\in C_0^\infty(U_i)$ so that $\tilde{\rho}_i(x)\in [0,1]$ and that $\tilde{\rho}_i(x)\equiv 1$ on the support of $\rho_i$. 
The transfer operators $
\L^t_{ij}:C^\infty(U_j)\to C^\infty_0(U_i)$ for $1\le i, j\le \ell$ and $t\in \real$ is then defined by
\[
\L^t_{ij} u(x)=g^t_{ij}(x)\cdot u(F^t_{ij}(x))
\]
where
\[
g^t_{ij}(x)=\rho_i(x)\cdot \tilde{\rho}_j(F^t_{ij}(x))\cdot g^t(\kappa^{-1}_i(x)).
\]
These transfer operators as a whole form the operator 
\[
\mathbf{L}^t:\bigoplus_{i=1}^\ell \, C^\infty_0(U_i)\to \bigoplus_{i=1}^\ell \, C^\infty_0(U_i),
\qquad
\mathbf{L}^t((u_i)_{i=1}^\ell)=\left(\sum_{j}\L^t_{ij}u_j\right)_{i=1}^\ell.
\]
By the definitions above,  we have the commutative  diagram of operators:
\begin{equation}\label{cdls}
\begin{CD}
\bigoplus_{i=1}^\ell C^\infty_0(U_i)@>{\mathbf{L}^t}>> \bigoplus_{i=1}^\ell C^\infty_0(U_i)\\
@A{\iota}AA @A{\iota}AA\\
C^\infty(M)@>{\L^t}>> C^\infty(M)
\end{CD}
\end{equation}
where
\[
\iota:C^\infty(M)\to \bigoplus_{i=1}^\ell C^\infty_0(U_i),\qquad 
\iota(u)=\big(\rho_i\cdot u\circ \kappa_i^{-1}\big)_{i=1}^\ell.
\]

In Section \ref{s:pFBI}, we will  introduce a scale of  Hilbert spaces $(\H^r_{\aniso},\|\cdot\|_r)$ for $r\in \real$ such that \[
H^r\subset H^r_{\aniso}\subset H^{-r},
\]
where $H^r$ is the Sobolev space  of order $r$ on $\real^{2d+1}$. 
And we will prove
\begin{theorem}
\label{th:localmain}
For any given $r>d$, there exists a constant $C_0>0$ such that, for any \hbox{$\lambda$-hyperbolic} contact diffeomorphism $F:U\to U'$ with $\lambda$ sufficiently large  and for any $g\in C_0^\infty(U)$, the  transfer operator 
\begin{equation}\label{eqn:lfg}
\L u=g\cdot (u\circ F)
\end{equation}
extends naturally to a bounded operator $
\L:\H^r_{\aniso}\to \H^r_{\aniso}$ and the essential spectral radius of the extension is bounded by
\begin{equation}\label{eqn:bound}
 C_0\cdot \max\{ \Lambda(F,g), \|g\|_\infty\cdot \lambda^{-r}\cdot \Delta(F,g)\},
\end{equation}
 where $\Lambda(F,g)$ and $\Delta(F,g)$ are defined  as
\[
\Lambda(F,g)=\max_{x\in \mathrm{supp}(g)}\;\left(  \frac{|g(x)|}{ \sqrt{|\det (DF_x|_{E^+})|}}\right)
\]
and
\[
\Delta(F,g)=\max_{x\in \mathrm{supp}(g)}\; \sqrt{|\det (DF_x|_{E^+})|},
\]
with setting $E^+=\{ (x_0,x^+,x^-)\in \real^{2d+1}\mid x^-=0\}$.
\end{theorem}

For a bounded linear operator $L:B\to B$ on a Banach space $B$, its essential operator norm $\|L\|_{\mathrm{ess}}$ and essential spectral radius $\rho_{\mathrm{ess}}(L)$ are respectively the infimum of its operator norm and its spectral radius  under perturbation by compact operators, that is, 
\begin{align*}
&\|L\|_{\mathrm{ess}}:=\inf\{ \|L-K\|\mid K:B\to B\mbox{ is a compact operator.}\},\\
&\rho_{\mathrm{ess}}(L):=\inf\{ \rho(L-K)\mid K:B\to B\mbox{ is a compact operator.}\}
\end{align*}
By definition the essential spectral radius of a bounded linear operator does not exceed its essential operator norm. 

We show that the claims of the main theorem except for the lower bound for the essential spectral radius follow from Theorem \ref{th:localmain}. Consider the (unique) norm on $C^\infty(M)$ such that the injection $\iota$ is an isometric injection into $\bigoplus_{i=1}^\ell H^r_{\aniso}$ with respect to it and let $H^r_{\aniso}(M)$ be the completion of $C^\infty(M)$ with respect to that norm. From the former claim of Theorem \ref{th:localmain},  the commutative diagram (\ref{cdls}) extends naturally  to 
\[
\begin{CD}
\bigoplus_{i=1}^\ell \H^r_{\aniso}@>{\mathbf{L}^t}>> \bigoplus_{i=1}^\ell \H^r_{\aniso}\\
@A{\iota}AA @A{\iota}AA\\
\H^r_{\aniso}(M)@>{\L^t}>> \H^r_{\aniso}(M)
\end{CD}
\]
in which $\iota$ is an isometric embedding. 
From the latter claim of Theorem \ref{th:localmain}, we see that the essential operator norm of $\mathbf{L}^t:\bigoplus_{i=1}^\ell \H^r_{\aniso}\to\bigoplus_{i=1}^\ell \H^r_{\aniso}$ is bounded by
\begin{equation}\label{eqn:bd}
C_0\cdot  \max_{i,j}\left(\max\{ \Lambda(F_{ij}^t,g_{ij}^t), \|g_{ij}^t\|_\infty\cdot\lambda_0^{-rt}\cdot \Delta(F_{ij}^t,g_{ij}^t)\}\right)
\end{equation}
and so is that of $\L^t:H^r_{\aniso}(M)\to H^r_{\aniso}(M)$. 
Note that (\ref{eqn:bd}) is bounded by
\[
C_0 \cdot \max\left\{ \max_{x\in M} \frac{|g^t(x)|}{\sqrt{\det DF^t_x|_{E_u}}}, \|g^t\|_{\infty}\cdot \lambda_0^{-rt}\cdot  \max_{x\in M} \sqrt{\det DF^t_x|_{E_u}}\right\}
\]
with possibly different constant $C_0$ and  that the latter term in $\max\{\cdot\}$ above is smaller than the former if $r$ is sufficiently large. Therefore, by multiplicative property of essential spectral radius, we conclude the estimate
\begin{align*}
\rho_{\mathrm{ess}}(\L^t|_{\H^r_{\aniso}(M)})&=\lim_{n\to \infty} \|\L^{nt}:\H^r_{\aniso}(M)\to \H^r_{\aniso}(M)\|^{1/n}_{\mathrm{ess}}\le \Lambda^t.
\end{align*}

\subsection{Affine transformations and diffeomorphisms preserving the standard contact form $\alpha_0$}\label{ss:localdiffeo}
For each point $c=(c_0, c^+, c^-)\in \real^{2d+1}$, let $A_c:\real^{2d+1}\to \real^{2d+1}$ be  the affine transformation
\begin{equation}\label{eqn:ac}
 A_c(x_0,x^+,x^-)=\left(x_0+c_0- c^+ \cdot x^-+ c^-\cdot  x^+, \; x^++c^+,  \; x^-+c^-\right), 
\end{equation}
which preserves the standard contact form $\alpha_0$. 
The  totality $\Af=\{A_c\}_{c\in \real^{2d+1}}$ of such transformations form a transformation group that acts on $\real^{2d+1}$ transitively. In the following, most of our constructions (including that of the anisotropic Sobolev spaces) will be invariant with respect to the action of the transformation group $\Af$.

Let  $F:U\to U'$ be a $C^\infty$ diffeomorphism between open subsets in $\real^{2d+1}$ preserving the standard contact form $\alpha_0$. Then it preserves also the vector field $\partial_{x_0}$ as it is characterized by the conditions (\ref{Reeb}). Thereby, setting $x_\dag=(x_1,x_2,\cdots, x_{2d})$ for $x=(x_0,x_1,\dots, x_{2d})$, we may write $F$ locally  as  
\begin{equation}\label{Fh}
F(x_0, x_\dag)=(x_0+f(x_\dag), F_\dag(x_\dag)),
\end{equation}
where $F_\dag:\real^{2d}\to\real^{2d}$ and $
f:\real^{2d}\to \real$ are a $C^\infty$ diffeomorphism and a $C^\infty$ function respectively. 
Let $\alpha_\dag$ and $\omega_\dag$ be  the differential forms on $\real^{2d}$   defined respectively by
\begin{align}\label{eqn:aomega}
\alpha_\dag&=\left( \sum_{j=1}^d x_j\cdot dx_{j+d}-x_{j+d}\cdot dx_j\right),\quad \omega_\dag=\frac{1}{2}d\alpha_\dag=\sum_{j=1}^d dx_j\wedge dx_{d+j}.
\end{align}
Then $F_\dag$ above preserves the symplectic form $\omega_\dag$ and the function $f$ is determined from $F$ by the relation 
\[
df=DF^*(dx_0)-dx_0=\alpha_\dag-F_{\dag}^*(\alpha_\dag)
\]
up to difference by a constant. 

Suppose $F(0)=0$ in addition. 
Then, for any $1\le i,j\le 2d$, we have 
\begin{equation}\label{eqn:basic}\frac{\partial}{\partial x_i} f(0)=0,\qquad \frac{\partial^2}{\partial x_i \partial x_j}f(0)=0.
\end{equation}
These relations (in particular the latter) are not very obvious but 
the proof is straightforward.  We refer \cite[Lemma 4.1]{Tsujii2009} for the detail. 
The property (\ref{eqn:basic})  implies that there exits a constant $C>0$ such that, for $z\in \real^{2d+1}$ sufficiently close to the origin $0$ and for  
$\xi=(\xi_0, \xi^+, \xi^-)\in \real^{2d+1}$, one has the estimate
\begin{align*}
\|{}^tDF_z(\xi)-{}^tDF_{0}(\xi)\|&\le \|{}^tDf_{z}(\xi_0)\|+\|{}^t(DF_\dag)_z(\xi_\dag)-{}^t(DF_\dag)_{0}(\xi_\dag)\|\\
 &\le C(|\xi_0|\cdot \|z\|^2+\|\xi_\dag\|\cdot \|z\|).
\end{align*}
The last estimate  is applicable to the germ of $F$ at each point, by means of changes of  coordinates by affine transformations in $\Af$. Thus we obtain the following proposition, which stands without the assumption $F(0)=0$.
\begin{proposition}\label{pp:basic}
If a $C^\infty$ diffeomorphism $F:U\to U'$ between open subsets in $\real^{2d+1}$ preserves the standard contact form $\alpha_0$ and if $K$ is a compact subset of $U$, there exists a constant $C>0$ such that 
\[
\|{}^tDF_y(\xi)-{}^tDF_{y'}(\xi)\|\le C(|\xi_0|\cdot\|y_\dag-y'_\dag\|^2+ \|\xi-\xi_0\cdot \alpha_0(F(y))\| \cdot \|y_\dag-y'_\dag\|)
\]
for any two points $y=(y_0,y_\dag)$, $y'=(y'_0, y'_\dag)$ in $K$ and any $\xi=(\xi_0, \xi_\dag)\in \real^{2d+1}$.
\end{proposition}

\section{The FBI transform}\label{s:FBI}

In this section, we introduce the FBI (Fourier-Bros-Iaglonitzer) transform and give a few basic facts related to it.  As the Fourier transform decomposes functions into simple wave functions, the FBI transform 
decomposes functions into Gaussian wave packets. 
We refer \cite[Ch.3]{Martinez} for general argument on the FBI transform. 
\subsection{FBI transform}
For a pair $(x,\xi)$ of points $x$ and\footnote{Maybe it is more natural to regard $\xi$ as an element of the dual space of $\real^{D}$.} $\xi$ in the $D$-dimensional Euclidean space $\real^{D}$, we consider a $C^\infty$ function 
\[
\phi_{x, \xi}:\real^{D}\to \complex, \qquad \phi_{x, \xi}(y)=a_D \cdot \exp\big(i\xi  (y-(x/2))-\|y-x\|^2/2\big)
\]
where $a_D=(2\pi)^{-D/2}\cdot \pi^{-D/4}$. 

\begin{remark}\label{remark:half}
In the textbook \cite{Martinez}, the function $\phi_{x,\xi}$ is defined similarly but with the term $i\xi  (y-(x/2))$ replaced by $i\xi  (y-x)$. Our definition above is  slightly more convenient  for our argument, though the difference is not essential. 
\end{remark}

The FBI transform $\FBI$ maps a function $u(x)$ on $\real^{D}$ to the function  
\[
\FBI u(x,\xi)=\int \overline{\phi_{x,\xi}(y)}\cdot  u(y) dy
\]
on $\real^{2D}=\real^D\oplus \real^D$. Its (formal) adjoint $\FBI^*$ maps a function $v(x,\xi)$ on $\real^{2D}$ to the function
\[
\FBI^* v(y)=\int \phi_{x,\xi}(y) \cdot v(x,\xi) dx d\xi
\]
on $\real^D$. The following are the basic properties of these transforms.
\begin{proposition}\label{pp24} {\rm (1)} 
The FBI transform $\FBI$ is a continuous linear operator from $\Sch(\real^{D})$ to $\Sch(\real^{2D})$, while $\FBI^*$  is a continuous linear operator from $\Sch(\real^{2D})$ to $\Sch(\real^{D})$, where $\Sch(\real^{D})$ denotes the Schwartz space on $\real^{D}$.\\
{\rm (2)} The composition $\FBI^*\circ \FBI:\Sch(\real^{D})\to \Sch(\real^{D})$ is the identity operator. Consequently the FBI transform $\FBI$ extends to an isometry from $L^2(\real^{D})$ to $L^2(\real^{2D})$.
\end{proposition}
\begin{proof}
The first claim can be proved by a straightforward argument. Below we prove the second. 
The Schwartz kernel of the operator $\FBI^*\circ \FBI$ is 
\[
\int \phi_{x,\xi}(y)  \cdot \overline{\phi_{x,\xi}(y')}\;  dx \, d\xi.
\]
Performing integration with respect to $\xi$ and $x$ in turn, we see that this equals
\[
a_D^2\cdot (2\pi)^{D}\cdot \delta(y-y')\cdot \int  \exp(-\|y-x\|^2) \, dx=
a_D^2\cdot (2\pi)^{D}\cdot \pi^{D/2} \delta(y-y')=\delta(y-y').
\]
Clearly this implies that $\FBI^*\circ \FBI$ is the identity operator. 
\end{proof}

Note that, if we put $v(x,\xi)=\FBI u(x,\xi)$ for $u\in \Sch(\real^{D})$, the latter claim of the proposition above implies the following expression  of $u$ as a superposition of the wave packets $\phi_{x,\xi}(\cdot)$:
\[
u(y)=\FBI^* v(y)=\int v(x,\xi)\cdot \phi_{x,\xi}(y)\, dx d\xi.
\]
  
\subsection{The projection operator $\Pj$}

\begin{proposition} The composition $\Pj:=\FBI\circ \FBI^*:\Sch(\real^{2D})\to \Sch(\real^{2D})$ extends to the orthogonal projection $\Pj:L^2(\real^{2D})\to L^2(\real^{2D})$ to the closed subspace 
\[
\FBI(L^2(\real^{D}))=\{ \FBI u\in L^2(\real^{2D})\mid u\in L^2(\real^{D})\}.
\]
\end{proposition}
\begin{proof}
From the definition of $\Pj$ and Proposition \ref{pp24}, we have $\Pj=\Pj^*$ and $\Pj\circ \Pj=\Pj$. 
It remains to prove that 
\[
\FBI(L^2(\real^{D}))=\{ u\in L^2(\real^{2D})\mid \Pj u=u\}.
\]
If $u\in L^2(\real^{2D})$ satisfies $\Pj u=u$, it belongs to $\FBI(L^2(\real^{D}))$ because  $u=\Pj u=\FBI(\FBI^* u)$. Conversely, if $u\in \FBI(L^2(\real^{D}))$, we can take $v\in L^2(\real^{D})$ such that $u=\FBI v$ and obtain $\Pj u=\FBI\circ \FBI^*\circ \FBI v=\FBI v=u$ from Proposition \ref{pp24} (2).
\end{proof}
The projection operator $\Pj$ above is an integral transform
\[
\Pj u(x,\xi)=\int K_{\Pj}(x,\xi;x',\xi')\, u(x',\xi')\, dx' d\xi'
\]
with the kernel
\begin{equation}
K_{\Pj}(x,\xi;x', \xi')= \int \overline{\phi_{x, \xi}(y)}\cdot \phi_{x',\xi'}(y) dy.
\end{equation}
If we perform the (Gaussian) integration in $K_{\Pj}(\cdot)$, we see 
\[
K_{\Pj}(x,\xi;x', \xi')=(2\pi)^{-D/2}\cdot \exp
\left(
- \frac{i\cdot \Omega((x,\xi), (x', \xi'))}{2}-\frac{\|x-x'\|^2}{4}-\frac{\|\xi-\xi'\|^2}{4}
\right)
\]
where $\Omega:\real^{2D}\oplus \real^{2D}\to \real$ is the standard symplectic form on $\real^{2D}=\real^{D}\oplus (\real^{D})^*$ defined  by 
$\Omega((x,\xi), (x', \xi'))=x\cdot \xi' -\xi\cdot x'$.

We may generalize the projection operator $\Pj$ to a slightly more general setting. 
Let $\omega$ be a symplectic form on an even dimensional Euclidean space $E$ and suppose that it is compatible\footnote{Compatibility implies  that the linear map $J:E\to E$ defined by the relation $\omega(x,y)=(x,Jy)$ satisfies $J\circ J=-\Id$.} with the Euclidean norm.
We define the integral transform $
P_\omega:\Sch(E)\to \Sch(E)$ by
\begin{equation}\label{eqn:pjg}
P_\omega u(z)=(2\pi)^{-D/2}\int \exp(-i \cdot \omega(z,z')/2-\|z-z'\|^2/4) \cdot u(z')\, dz', 
\end{equation}
replacing $\Omega$ in the expression of $\Pj$  above by $\omega$. Then we have
\begin{proposition}
$P_\omega$ extends to an orthogonal projection operator in $L^2(E)$.
\end{proposition}
\begin{proof} We can check $P_\omega\circ P_\omega=P_\omega$ and $P_\omega^*=P_\omega$ by a straightforward computation. (Or, one can introduce an orthogonal coordinate system on $E$ in which $\Omega=\omega$.)
\end{proof}

\subsection{The action of affine transformations}\label{ss:affine}
We are going to consider the action of affine transformations viewed through the FBI transform. 
Let $B:\real^{D}\to \real^{D}$ be an affine transformation. 
We write its natural action on $T^*\real^D=\real^D\oplus \real^D$ as 
\[
\widetilde{B}:\real^{D}\oplus \real^{D}\to \real^{D}\oplus \real^{D},\quad 
\widetilde{B}(x,\xi)=(Bx, {}^t(DB)^{-1}\xi),
\]
where $DB$ denotes the linear part of $B$. 
We consider the pull-back operator by $B$,  
\[
\L_B:\mathcal{S}(\real^{D})\to \mathcal{S}(\real^{D}),\qquad \L_B u(x)= u(B(x)).
\]
The corresponding action of $B$ on the functions on $T^*\real^D=\real^D\oplus \real^D$ is\footnote{The coefficient $e^{-i B^{-1}(0)\cdot \xi/2}$ appears as the result of our definition of $\phi_{x,\xi}$. See Remark \ref{remark:half}. }
\[
\widetilde{\L}_B:\mathcal{S}(\real^{D}\oplus \real^{D})\to \mathcal{S}(\real^{D}\oplus \real^{D}),\quad 
\widetilde{\L}_B u(x,\xi)= e^{-i B^{-1}(0)\cdot \xi/2}\cdot u(\widetilde{B}(x,\xi)).
\]
\begin{lemma}\label{FBIiso}
If $B:\real^{D}\to \real^{D}$ is an isometry, we have 
\[
\widetilde{\L}_B\circ \FBI=\FBI\circ \L_B, \qquad \L_B\circ \FBI^*=\FBI^*\circ \widetilde{\L}_B\qquad \Pj\circ \widetilde\L_B=\widetilde{\L}_B\circ \Pj.
\]
\end{lemma}
\begin{proof}The first and second equality follows from the relation $\L_B \phi_{\widetilde{B}(x,\xi)}=\phi_{x,\xi}$. The last follows from these and the definition $\Pj=\FBI\circ \FBI^*$. 
\end{proof}

If the affine transformation $B:\real^{D}\to \real^{D}$ is not an isometry, the lemma above is no longer true and we need some modification. Below we assume that $B$ is a linear transformation, {\it i.e. }$B(0)=0$ and that $\det B=1$, for simplicity. 

If we define the operator $
\widehat{\L}_B: \mathcal{S}(\real^{2D})\to \mathcal{S}(\real^{2D})$ by
\begin{equation}\label{eqn:liftB}
\widehat{\L}_B=\FBI \circ \L_B\circ \FBI^*, 
\end{equation}
it makes the following diagram commutes:
\begin{equation}\label{eqn:cdb}
\begin{CD}
\mathcal{S}(\real^{D}\oplus \real^{D})@>{\widehat{\L}_B}>> \mathcal{S}(\real^{D}\oplus \real^{D})\\
@A{\FBI}AA @A{\FBI}AA\\
\mathcal{S}(\real^{D})@>{\L_B}>> \mathcal{S}(\real^{D})\\
\end{CD}
\end{equation}
So it should be natural to call $\widehat{\L}_B$ the lift of $\L_B$ with respect to the FBI transform. 
For the linear transformation $B$, we set
\begin{equation*}
d(B)=\det\big((\Id+{}^tB\cdot B)/2\big)^{1/2}.
\end{equation*}
Then we have the following expression for the lift $\widehat{\L}_B$. 
\begin{proposition}\label{FBIaffine}
$\displaystyle 
\widehat{\L}_B=d(B)\cdot \Pj\circ \widetilde{\L}_B\circ \Pj
$.
\end{proposition}
\begin{proof}
The operator $\FBI^*\circ \widetilde{\L}_B\circ  \FBI$ can be written as an integral operator
\[
(\FBI^*\circ \widetilde{\L}_B\circ  \FBI) u(y)=\int K(y,y') \, u(y')\, dy'
\]
with the kernel
\[
K(y,y')=a_D^2\cdot \int e^{i \xi \cdot (y-B^{-1} y')-|y-x|^2/2-|y'-B x|^2/2 } dx d\xi. 
\]
If we calculate the integration using a change of variable $z=x-B^{-1}y'$, we obtain
\begin{align*}
&K(y,y')=\pi^{-D/2} \cdot \delta(y-B^{-1} y')\cdot  
\int e^{-|B^{-1} y'-x|^2/2-|y'-B x|^2/2 } dx\\
&=\pi^{-D/2} \cdot \delta(y-B^{-1} y') 
\int e^{-|z|^2/2-|B z|^2/2 } dz= d(B)^{-1}\cdot  \delta( y'-By).
\end{align*}
This implies  $\L_B = d(B)\cdot\FBI^*\circ \widetilde{\L}_B\circ  \FBI$. Composing $\FBI$ and $\FBI^*$ from the left and right respectively, we obtain the required formula. 
\end{proof}

\subsection{Change of variables}\label{ss:sep}
In this subsection, we set $D=2d$ and consider a linear transformation $B:\real^{2d}\to\real^{2d}$ that preserves the symplectic form~$\omega_\dag$ on~$\real^{2d}$ defined in (\ref{eqn:aomega}). Below we show that the operator $\widehat{\L}_B$ can be identified with the tensor product of two (almost) identical operators through an appropriate change of variables. 
This argument is essentially due to F.~Faure\cite{Faure07}.

To begin with, note that the Hilbert space $L^2(\real^{2d}\oplus \real^{2d})$ is naturally identified with the
tensor product $L^2(\real^{2d})\otimes L^2(\real^{2d})$. Let us consider the linear bijection 
\begin{equation}\label{eqn:Z}
Z:\real^{2d}\oplus \real^{2d}\to \real^{2d}\oplus \real^{2d}, 
\quad Z(x,\xi)=(2^{-1/2}(\xi+Jx), 2^{-1/2}(\xi-Jx))
\end{equation}
as a coordinate change, where  $J:\real^{2d}\to \real^{2d}$ is the linear map characterized by the condition $(y, Jx)=\omega_\dag(y,x)$ or, more concretely, defined by
\[
J(x^+,x^-)=(x^-, -x^+)\quad \mbox{ for }x^+, x^-\in \real^{d}.
\]
From the assumption that $B$ preserves $\omega_\dag$, 
the diagram
\[
\begin{CD}
 \real^{2d}\oplus \real^{2d}@>{\widetilde{B}}>>  \real^{2d}\oplus \real^{2d}\\
 @V{Z}VV @V{Z}VV\\
 \real^{2d}\oplus \real^{2d} @>{{}^tB^{-1}\oplus{}^tB^{-1}}>>  \real^{2d}\oplus \real^{2d}
\end{CD}
\]
commutes.
Consequently, for the unitary operators 
\[
Z^*:L^2(\real^{2d}\oplus \real^{2d})\to L^2(\real^{2d}\oplus \real^{2d}),\qquad 
Z^*u(x,\xi)=u(Z(x,\xi)),
\]
and 
\begin{equation*}
\L_0:L^2(\real^{2d})\to L^2(\real^{2d}),\qquad 
\L_0u(\xi)=u({}^tB^{-1}\xi),
\end{equation*}
the following  diagram commutes:
\[
\begin{CD}
 L^2(\real^{2d}\oplus \real^{2d})@>{\widetilde{\L}_B}>>  L^2(\real^{2d}\oplus \real^{2d})\\
 @A{Z^*}AA @A{Z^*}AA\\
 L^2(\real^{2d}\oplus \real^{2d})@>{\L_0\otimes \L_0}>>  L^2(\real^{2d}\oplus \real^{2d})
\end{CD}
\]

Another  important property of $Z$ is that it is an isometry with respect to the standard Euclidean norm on $\real^{2d}\oplus \real^{2d}$ and intertwines the standard symplectic form $\Omega$ with $\omega_\dag\oplus (-\omega_\dag)$. That is, for $
Z(x,\xi)=(z,w)$, $Z(x',\xi')=(z',w')$, we have
\[
((x, \xi), (x', \xi'))_{\real^{2d}\oplus \real^{2d}}=(z,z')_{\real^{2d}}+(w,w')_{\real^{2d}}
\]
and
\[
\Omega((x, \xi), (x', \xi'))=\omega_0(z,z')-\omega_0(w,w').
\]

Let $\Pj_{0}:L^2(\real^{2d})\to L^2(\real^{2d})$ be the projection operator  defined by (\ref{eqn:pjg}) with the setting $D=d$ and $\omega=\omega_\dag$. Then, from the property of $Z$ noted in the last paragraph,  we see that the following diagram also commutes:
\[
\begin{CD}
L^2(\real^{2d}\oplus \real^{2d})@>{\Pj}>> L^2(\real^{2d}\oplus \real^{2d})\\
@A{Z^*}AA @A{Z^*}AA\\
L^2(\real^{2d}\oplus \real^{2d}) @>{\Pj_0\otimes \overline{\Pj_0}}>>L^2(\real^{2d}\oplus \real^{2d})
\end{CD}
\]
Therefore, if we define the operator $\widehat{\L}_0:L^2(\real^{2d})\to L^2(\real^{2d})$ by
\begin{equation}\label{eqn:whl}
\widehat{\L}_0=d(B)^{1/2}\cdot  \Pj_0\circ \L_0\circ \Pj_0, 
\end{equation}
we have the commutative diagram 
\begin{equation}\label{CD:lb}
\begin{CD}
L^2(\real^{2d}\oplus \real^{2d})@>{\widehat{\L}_B}>> L^2(\real^{2d}\oplus \real^{2d})\\
@A{Z^*}AA @A{Z^*}AA\\
L^2(\real^{2d}\oplus \real^{2d})  @>{\widehat{\L}_0\otimes \overline{\widehat{\L}_0}}>>L^2(\real^{2d}\oplus \real^{2d})
\end{CD}
\end{equation}
In conclusion,  the  operator $\widehat{\L}_B$ is identified with the tensor product of the operator $\widehat{\L}_0$ and its complex conjugate through the coordinate change by $Z$.

\section{Definition of anisotropic Sobolev spaces}\label{s:pFBI}
In this section, we introduce what we call the partial FBI transform  and then give the definition of  anisotropic Sobolev spaces using it. 
\subsection{The partial FBI transform} \label{ss:pFBI}
The partial FBI transform on $\real^{2d+1}$ is  a combination of the Fourier transform in the first coordinate and the FBI transform in the other coordinates with some scaling. 
Below we give  a precise definition for it. 

Take and fix a $C^\infty$ function $\chi:\real\to [0,1]$ such that
\[
\chi(s)=\begin{cases}
1, &\mbox{if $s\le 4/3$};\\
0, &\mbox{if $s\ge 5/3$}.
\end{cases}
\]
For a real number $s$, we set\footnote{In most of the literature, $\langle s\rangle$ is defined to be $(1+s^2)^{1/2}$. But the definition here is more convenient for our argument. } 
\[
\langle s\rangle=|s|\cdot (1- \chi(|s|))+\chi(|s|),
\]
so that $\langle s\rangle\ge 1$ for any $s$ and that $\langle s\rangle=|s|$ holds if $|s|\ge 2$. For a given point $y=(y_i)_{i=0}^{2d} \in \real^{2d+1}$, we will write
\[
y^+=(y_1,y_2,\dots, y_d), \quad y^-=(y_{d+1}, y_{n+2}, \dots, y_{2d}) \quad \mbox{and}\quad y_\dag=(y^+,y^-).
\]
As in the definition of the FBI transform, we introduce a family of functions 
\[
\Phi_{x_\dag,\xi}:\real^{2d+1}\to \complex\quad \mbox{ for $(x_\dag,\xi)\in \real^{2d}\oplus \real^{2d+1}$}
\]
 defined by 
\begin{align*}
\Phi_{x_\dag,\xi}(y)&=\frac{\langle \xi_0\rangle^{d/2}\cdot a_{2d}}{(2\pi)^{1/2}}\cdot
\exp\left(i\xi_0 y_0+i\xi_\dag(y_\dag-(x_\dag/2))-\langle \xi_0\rangle \|y_\dag-x_\dag\|^2/2\right)\\
&=\left(\frac{1}{(2\pi)^{1/2}}\cdot
e^{i\xi_0 y_0}\right)
\cdot \left(
\langle \xi_0\rangle^{d/2}\cdot a_{2d} \cdot
e^{i\xi_\dag(y_\dag-(x_\dag/2))-\langle \xi_0\rangle \|y_\dag-x_\dag\|^2/2}\right).
\end{align*}
The partial FBI transform $\pFBI$ maps a function $u(y)$ on  $\real^{2d+1}$ to the function 
\[
\pFBI  u(x_\dag,\xi)=\int \overline{\Phi_{x_\dag,\xi}(y)}\cdot  u(y)\, dy\qquad \mbox{on}\quad\real^{2d}\oplus \real^{2d+1}.
\]
And its formal adjoint $\pFBI^*$ maps a function $u(x_\dag,\xi)$ on  $\real^{2d}\oplus \real^{2d+1}$ to the function
\[
\pFBI^*u(y)=\int \Phi_{x_\dag,\xi}(y)\, u(x_\dag,\xi) \; dx_\dag d\xi\qquad \mbox{on}\quad  \real^{2d+1}.
\]

The partial FBI transform may be viewed as a combination of the Fourier transform and the FBI transform with some scaling as follows:
Consider the Fourier transform in the first variable, 
\[
\Fourier_0:\Sch(\real^{2d+1})\to \Sch(\real^{2d+1}), \qquad \Fourier_0u(\xi_0, y_\dag)=
(2\pi)^{-1/2}\int e^{-i\xi_0 y_0} u(y_0, y_\dag) dy_0
\]
and the FBI transform in the other coordinates,
\[ 
\FBI_\dag:\Sch(\real^{2d+1})\to \Sch(\real^{2d}\oplus \real^{2d+1}), \qquad \FBI_\dag u(x_\dag, \xi)=  \int \overline{\phi_{x_\dag,\xi_\dag}(z_\dag)}\cdot  u(\xi_0,z_\dag) dz_\dag,
\]
where $\xi=(\xi_0,\xi_\dag)\in \real^{2d+1}=\real\oplus \real^{2d}$. 
Also, consider the  operators 
\[
S:\Sch(\real^{2d+1})\to \Sch(\real^{2d+1}),\qquad Su(\xi_0, x_\dag)= \langle \xi_0\rangle^{-d/2}\cdot  u(\xi_0, \langle \xi_0\rangle^{-1/2} x_\dag)
\]
and
\[
\widehat{S}:\Sch(\real^{2d}\oplus \real^{2d+1})\to \Sch(\real^{2d}\oplus\real^{2d+1}),\quad \widehat{S}u(x_\dag, \xi)=  u\left( \langle \xi_0\rangle^{-1/2} x_\dag, ( \xi_0, \langle \xi_0\rangle^{1/2} \xi_\dag)\right)
\]
associated to the scaling  
\[
\real\oplus\real^{2d}\ni (\xi_0,x_\dag)\mapsto \big(\xi_0, \langle \xi_0\rangle^{-1/2}x_\dag\big )\in\real\oplus\real^{2d}.
\]
Then the FBI transform $\pFBI$ and its formal adjoint $\pFBI^*$ are expressed as
\[
\pFBI=(\widehat{S}^{-1}\circ \FBI_\dag\circ  S)\circ \Fourier_0,\qquad \pFBI^*=\Fourier_0^{-1}\circ (S^{-1}\circ \FBI_\dag^*\circ  \widehat{S}).
\]
From this expression and the properties of the Fourier and FBI transform, we obtain 
\begin{proposition}\label{pp24d}
{\rm (1)}
The partial FBI transform $\pFBI$ is a continuous linear operator from $\Sch(\real^{2d+1})$ to $\Sch(\real^{2d}\oplus \real^{2d+1})$, while its formal adjoint $\pFBI^*$ is a continuous linear operator from $\Sch(\real^{2d}\oplus \real^{2d+1})$ to $\Sch(\real^{2d+1})$. \\
{\rm (2)}
The composition $\pFBI^*\circ \pFBI:\Sch(\real^{2d+1})\to \Sch(\real^{2d+1})$ is the identity operator. Consequently  $\pFBI$ extends to an isometric embedding of $L^2(\real^{2d+1})$ into $L^2(\real^{2d}\oplus \real^{2d+1})$.
\end{proposition}

\begin{proposition} The composition $\pPj:=\pFBI\circ \pFBI^*:\Sch(\real^{2d}\oplus \real^{2d+1})\to \Sch(\real^{2d}\oplus \real^{2d+1})$ extends to the orthogonal projection 
$\pPj:L^2(\real^{2d}\oplus \real^{2d+1})\to L^2(\real^{2d}\oplus \real^{2d+1})$ to the closed subspace 
$\pFBI(L^2(\real^{2d+1}))=\{ \pFBI u\mid u\in L^2(\real^{2d+1})\}$.
\end{proposition}

\subsection{Anisotropic Sobolev spaces}\label{s:anisoH}
In this subsection, we define the anisotropic Sobolev space $H^r_\aniso$ as the pull-back of some weighted $L^2$ space on $\real^{2d}\oplus \real^{2d+1}$ by the partial FBI transform $\pFBI$.

For $\theta>0$, let $\cone^*_+(\theta)$ and $\cone^*_-(\theta)$ be the cones in $\real^{2d}$ defined by 
\begin{align*}
\cone^*_+(\theta)&=\{(\zeta^+, \zeta^-)\in \real^{2d}=\real^d\oplus \real^d\mid \|\zeta^-\|\le \theta\|\zeta^+\|\},\quad \mbox{and}\\
\cone^*_-(\theta)&=\{(\zeta^+, \zeta^-)\in \real^{2d}=\real^d\oplus \real^d\mid \|\zeta^+\|\le \theta\|\zeta^-\|\}.
\end{align*}
We take $C^\infty$ functions 
$\psi_\sigma:\mathbf{P}\real^{2d}\to [0,1]$, $\sigma=\pm$,  on the projective space $\mathbf{P}\real^{2d}$ such that \begin{align*}
\psi_+([\zeta])+\psi_-([\zeta])=1\quad\mbox{and }\quad
\begin{cases}
\psi_+([\zeta])=1&\quad \mbox{ if $\zeta\in \cone^*_+(1/3)$,}\\
\psi_-([\zeta])=1&\quad \mbox{ if $\zeta\in \cone^*_-(1/3)$,}
\end{cases}
\end{align*}
where $[\zeta]$ denotes the element in $\mathbf{P}\real^{2d}$ that is represented by $\zeta\in \real^{2d}$, and then  introduce the function
\[
W^r_{\aniso}:\real^{2d}\to \real, \quad W^r_{\aniso}(\zeta)=
\psi_+([\zeta])\cdot \langle \|\zeta\|\rangle^{-r}+\psi_-([\zeta])\cdot \langle \|\zeta\|\rangle^{+r}.
\]
We define the weight function $
\weight^{r}_{\mathrm{aniso}}:\real^{2d}\oplus \real^{2d+1}\to \real_+$ as follows: 
In the case $x=0$, we set
\[
\weight^{r}_{\mathrm{aniso}}(0,\xi)
=W^{2r}_{\mathrm{aniso}}\left(
 \frac{\xi_\dag}
 {\langle \|\xi\|\rangle^{1/2}} \right) \qquad\mbox{for $\xi=(\xi_0,\xi_\dag)\in \real^{2d+1}=\real\oplus \real^{2d}$.}
\]
 (Notice that we have $2r$ in the superscript of $W^{2r}_{\mathrm{aniso}}(\cdot)$ on the right hand side.)
 Then we  extend this definition to the case $x\neq 0$ uniquely so that it is invariant with respect to the natural action of the transformation group $\Af$. In other words, we set
\begin{align*}
&\weight^{r}_{\mathrm{aniso}}(x_\dag,\xi)=\weight^{r}_{\mathrm{aniso}}(0,{}^tDA_{(0,x_\dag)}(\xi)).
 \end{align*}
Note that $A_{(0,x_\dag)}$ is defined in (\ref{eqn:ac}) with setting $c=(0,x_\dag)\in \real^{2d+1}$, so that
\[
{}^tDA_{(0,x)}(\xi_0, \xi_\dag)=(\,\xi_0\,,\, \xi^{+}+ \xi_0\cdot  x^{-}\,, \, \xi^{-} -\xi_0\cdot x^+)
=(\,\xi_0, \,\xi_\dag+ \xi_0\cdot  J(x)\,).
\]

Now we define the anisotropic Sobolev space $H^{r}_{\aniso}$ as the completion of the Schwartz space $\Sch(\real^{2d+1})$ with respect to the norm 
\[
\|u\|_{r}=\left\|\weight^{r}_{\mathrm{aniso}}\cdot \pFBI u\right\|_{L^2(\real^{2d}\oplus \real^{2d+1})}.
\]
By definition, the partial FBI transform $\pFBI$ extends to the isometric embedding
\[
\pFBI:H^{r}_{\aniso}\to L^2(\real^{2d}\oplus \real^{2d+1};\weight^{r}_{\mathrm{aniso}})
\]
where $L^2(\real^{2d}\oplus \real^{2d+1};\weight^{r}_{\mathrm{aniso}})$ denotes the weighted $L^2$ space with weight $\weight^{r}_{\mathrm{aniso}}$.   

We give a relation between the anisotropic Sobolev spaces introduced above and the usual Sobolev spaces. 
(The proofs of  Lemma \ref{lm:sob} and Corollary \ref{cor:sob} below will be given in the appendix.) 
Recall that the Sobolev space $H^r$ of order $r$ on $\real^{2d+1}$ is defined as the completion of the Schwartz space $\Sch(\real^{2d+1})$ with respect to the norm
\[
\|u\|_{H^r}=\|\langle\xi\rangle^{r}\cdot \Fourier u(\xi)\|_{L^2(\real^{2d+1})}
\]
where $\Fourier$ denotes the Fourier transform. For  another norm 
\[
\|u\|'_{H^r}=\|\langle\xi\rangle^{r}\cdot \pFBI u(x,\xi)\|_{L^2(\real^{2d}\oplus \real^{2d+1})}
\]
defined by using the partial FBI transform $\pFBI$, one can show 
\begin{lemma}\label{lm:sob}
The two norms $\|\cdot \|_{H^r}$ and $\|\cdot \|'_{H^r}$ on $\Sch(\real^{2d+1})$ are equivalent. 
\end{lemma}
For a subset $K\subset \real^{2d+1}$, let 
$C^\infty(K)$ be the set of $C^\infty$ functions whose supports are  contained in $K$, and let $H^r_\aniso(K)$ (resp. $H^r(K)$) be the closure of  $C^\infty(K)$  in $H_\aniso^r$ (resp. $H^r$).
As a consequence of the last lemma, one obtains
\begin{corollary}\label{cor:sob}
For any compact subset $K\subset \real^{2d+1}$, we have 
\[
H^{r}(K)\subset H_{\aniso}^{r}(K)\subset H^{-r}(K).
\]
\end{corollary}

\subsection{The action of linear transformations on $H^r_\aniso$.}\label{ss:l0}
Let us consider a linear transformation
\[
\mathrm{Id}\oplus B:\real^{2d+1}\to \real^{2d+1}, \quad (\mathrm{Id}\oplus B)(x_0, x_\dag)=(x_0, B(x_\dag))
\]
where $B:\real^{2d}\to \real^{2d}$ is a linear transformation satisfying the following hyperbolicity conditions for  some large $\lambda\gg 1$:
\begin{itemize}
\setlength{\itemsep}{6pt plus 1pt minus 1pt}

\item[(B1)] $B(\real^{2d}\setminus \cone^*_-(1/10))\subset \cone^*_+(1/10)$, $B^{-1}(\real^{2d}\setminus \cone^*_+(1/10))\subset \cone^*_-(1/10)$, 
\item[(B2)] $\|B(v)\|\ge \lambda \|v\|$ if $v\in \real^{2d}\setminus \cone^*_-(1/10)$, and
\item[(B3)] $\|B^{-1}(v)\|\ge \lambda \|v\|$ if $v\in \real^{2d}\setminus \cone^*_+(1/10)$.
\end{itemize}
Below we study the pull-back operator 
\[
\L_{\mathrm{Id}\oplus B}u=u\circ (\mathrm{Id}\oplus B)
\]
acting on the anisotropic Sobolev space $H^r_\aniso$, as a simple model of the transfer operator (\ref{eqn:lfg}). To this end, we introduce the operator 
\[
\widehat{\L}_{\mathrm{Id}\oplus B}=\pFBI\circ \L_{\mathrm{Id}\oplus B}\circ \pFBI^*
\]
which makes the following diagram commutes:
\[
\begin{CD}
\Sch(\real^{2d}\oplus \real^{2d+1})@>{\widehat{\L}_{\mathrm{Id}\oplus B}}>> \Sch(\real^{2d}\oplus \real^{2d+1})\\
@A{\pFBI}AA @A{\pFBI}AA \\
\Sch(\real^{2d+1})@>{\L_{\mathrm{Id}\oplus B}}>> \Sch(\real^{2d+1})
\end{CD}
\]
In order to see that the operator $\L_{\mathrm{Id}\oplus B}$ induces a bounded operator on 
$H^r_\aniso$, it is enough to check that the lift $\widehat{\L}_{\mathrm{Id}\oplus B}$ extends to a bounded operator on the weighted $L^2$ space $L^2(\real^{2d}\oplus \real^{2d+1};\weight^r_\aniso)$ with the norm $\|u\|:=\|\weight^r_\aniso\cdot u\|_{L^2}$. 

Recall that the partial FBI transform is a combination of the Fourier transform in the flow direction and the FBI transform with some scaling in the transversal directions. Since the map $\mathrm{Id}\oplus B$ preserves the frequency in the flow direction, we can separate the actions of  
$\widehat{\L}_{\mathrm{Id}\oplus B}$ into each frequency. Thus, taking the scaling in the transversal direction into account,  we see that the operator norm of the lift $\widehat{\L}_{\mathrm{Id}\oplus B}$ on $L^2(\real^{2d}\oplus\real^{2d+1};\weight^r_\aniso)$ equals the  supremum of the operator norms of
\begin{equation}\label{eqn:lbb}
\widehat{\L}_B:L^2(\real^{2d}\oplus\real^{2d};\mathcal{W}_{\langle \xi_0\rangle})\to L^2(\real^{2d}\oplus\real^{2d};\mathcal{W}_{\langle \xi_0\rangle})
\end{equation}
for $\xi_0\in \real$,  where $\widehat{\L}_B$ is defined in (\ref{eqn:liftB}) and $\mathcal{W}_s:\real^{2d}\oplus\real^{2d}\to \real$ is defined by
\begin{equation}\label{wss}
\mathcal{W}_s(x_\dag,\xi_\dag)=W^{2r}_\aniso\left(\frac{\xi_\dag+Jx_\dag}{\langle (1+s^{-1}\cdot \|\xi_\dag+Jx_\dag\|^2)^{1/2}\rangle^{1/2}}\right).
\end{equation}
Next recall the change of variables discussed in Subsection \ref{ss:sep}, in particular, the commutative diagram (\ref{CD:lb}). 
Since we have
\begin{equation}\label{ws}
\mathcal{W}_s\circ Z^{-1}(z,w)=\mathcal{V}_s(z):=W^{2r}_\aniso\left(\frac{z}{\langle (1+s^{-1}\cdot \|z\|^2)^{1/2}\rangle^{1/2}}\right),
\end{equation}
the operator (\ref{eqn:lbb}) is identified with the tensor product of 
\[
\widehat{\L}_0:L^2(\real^{2d}, \mathcal{V}_s)\to L^2(\real^{2d},\mathcal{V}_s)
\quad\mbox{and}\quad
\overline{\widehat{\L}_0}:L^2(\real^{2d})\to L^2(\real^{2d})
\]where $\widehat{\L}_0$ is that defined in Subsection \ref{ss:sep}. And, for these two operators, we  show\begin{lemma}\label{lm:linFBI}
The operator $\widehat{\L}_0$ extends naturally to bounded operators both on the Hilbert spaces $L^2(\real^{2d})$ and $L^2(\real^{2d}, \mathcal{V}_s)$ for $s\ge 1$. Further,
\begin{itemize}
\setlength{\itemsep}{6pt plus 1pt minus 1pt}
\item[{\rm (1)}] the operator norm of $\widehat{\L}_0:L^2(\real^{2d})\to L^2(\real^{2d})$ is  $1$, and
\item[{\rm (2)}] the operator norm of $\widehat{\L}_0:L^2(\real^{2d}, \mathcal{V}_s)\to L^2(\real^{2d},\mathcal{V}_s)$ is bounded  by
\[
C_0\cdot  \max\{d(B)^{-1/2}, \,d(B)^{1/2}\cdot \lambda^{-r}\}
\]
 where $C_0$ is a constant that does not depend on $B$ nor $s\ge 1$.
\end{itemize}
\end{lemma}
Once we prove this lemma, we conclude that the operator norm of $\L_{\mathrm{Id}\oplus B}$ on $H^r_\aniso$ is bounded by $C_0\cdot  \max\{d(B)^{-1/2}, \,d(B)^{1/2}\cdot \lambda^{-r}\}$. 
Notice that  the last quantity corresponds to (\ref{eqn:bound}) in Theorem \ref{th:localmain} since $d(B)$ is proportional to $\det B|_{E^+}$. In the following sections, we will prove Theorem \ref{th:localmain} by reducing it to this simple case.

\begin{proof}[Proof of Lemma \ref{lm:linFBI}]
The claim (1) follows, for instance, from the fact that the operator norm of 
$\widehat{\L}_0\otimes \widehat{\L}_0$ on $L^2(\real^{2d}\oplus\real^{2d})$ equals 
that of $\widehat{\L}_B$ on $L^2(\real^{2d}\oplus\real^{2d})$ and hence equals $1$, as we showed in Subsection \ref{ss:sep}. 
Below we prove the claim (2). We will use $C_0$ as a generic symbol for constants that do not depend on $B$ nor $s$.

From the definition, the operator $\widehat{\L}_0$ can be  written as an integral operator
\[
\widehat{\L}_0u(z)=\int K(z,z')\, u(z')\, dz'
\]
and the kernel satisfies
\begin{align*}
|K(z,z')|&\le (2\pi)^{-d} \cdot d(B)^{1/2}\cdot \int \exp(-\|z-w\|^2/4-\|{}^tB^{-1}w-z'\|^2/4)dw.
\end{align*}
Hence, to prove the claim (2),  it is enough to show 
\[
\| \mathcal{Q}\circ \L_0\circ \mathcal{Q}:L^2(\real^{2d}, \mathcal{V}_s)\to L^2(\real^{2d}, \mathcal{V}_s)\|\le C_0\cdot \max\{d(B)^{-1}, \,\lambda^{-r}\}
\]
where $\mathcal{Q}$ is the convolution operator
\[
\mathcal{Q} u(z)=\phi*u(z)\quad \mbox{with \;\;$\phi(z)=\exp(-\|z\|^2/4)$.}
\]

By using the rapidly decaying property of $\phi$ and the definition of $\mathcal{V}_s$, 
we have
\begin{equation}\label{eqn:v}
\mathcal{V}_s(z)\cdot \phi(z-z')\le C_0\cdot \mathcal{V}_s(z')\cdot \langle \|z-z'\|\rangle^{-2d-1}
\end{equation}
where $C_0$ is a constant that does not depend on  $s$, provided that we take appropriate functions $\psi_{\pm}$ in the definition of $W^t_\aniso(\cdot)$. (See Remark \ref{rem:v} below.) This estimate implies in particular
\[
\|\mathcal{Q}:L^2(\real^{2d}, \mathcal{V}_s)\to L^2(\real^{2d}, \mathcal{V}_s)\|\le C_0.
\] 
\begin{remark}\label{rem:v}
To have the inequality (\ref{eqn:v}) hold, it may be necessary to put a  technical condition on the functions $\psi_{\pm}$ to avoid pathological cases, though we do not know whether it is really necessary. For instance, if we assume the condition that the first derivatives of the functions $\xi\mapsto (\psi_{\pm})^\epsilon ([\xi])$ is bounded on the unit sphere for each $\epsilon>0$, which can be fulfilled easily, we can show that the first derivatives of $\log \mathcal{V}_s$ is bounded by a constant independent of $s\ge 1$ and hence the inequality  (\ref{eqn:v}) follows. 
\end{remark}

Let $\mathcal{E}$ be the ellipsoid in $\real^{2d}$ defined by the condition
\[
|{}^tz\cdot (I+B^{-1}\cdot {}^tB^{-1})^{-1}\cdot z|\le 1.
\]
From the definition of $\mathcal{V}_s$ and hyperbolicity of $B$, it holds
\[
\mathcal{V}_s(B^t z)\le C_0\cdot \lambda^{-r}\cdot \mathcal{V}_s(z)\qquad \mbox{for $z\notin \mathcal{E}$}
\]
and therefore we have
\begin{align}\label{eqn:q1}
\|\mathcal{Q}\circ \L_0\circ &(\mathrm{Id}-\mathbf{1}_{\mathcal{E}})\circ \mathcal{Q}:L^2(\real^{2d}, \mathcal{V}_s)\to L^2(\real^{2d}, \mathcal{V}_s)\|\\
&\le C_0\|\L_0\circ (\mathrm{Id}-\mathbf{1}_{\mathcal{E}}):L^2(\real^{2d}, \mathcal{V}_s)\to L^2(\real^{2d}, \mathcal{V}_s)\|\le C_0\cdot \lambda^{-r}\notag
\end{align}
where $\mathbf{1}_{\mathcal{E}}$ denotes the multiplication by the characteristic function  of ${\mathcal{E}}$.

The operator norm of the remainder part 
\begin{align}\label{eqn:q2}
\mathcal{Q}\circ \L_0\circ \mathbf{1}_{\mathcal{E}}\circ \mathcal{Q}:L^2(\real^{2d}, \mathcal{V}_s)\to L^2(\real^{2d}, \mathcal{V}_s)
\end{align}
equals that of the integral operator
\[
\L':L^2(\real^{2d})\to L^2(\real^{2d}), \quad \L'u(z)=\int k(z,z')\, u(z') \,dz'
\]
with the kernel
\[
k(z,z')=
\int_{\mathcal{E}}\frac{\mathcal{V}_s(z)}{\mathcal{V}_s(z')}\cdot  \phi(z-{}^tBw)\cdot \phi(w-z') \, dw. 
\]
Once we prove the estimates
\begin{equation}\label{eqn:kzz}
\sup_{z'} \int k(z,z')\, dz\le C_0\cdot d(B)^{-1} \quad \mbox{and}\quad 
\sup_{z} \int k(z,z')\, dz'\le C_0\cdot d(B)^{-1},
\end{equation}
the Schur test\cite[Lemma 18.1.12]{Hormander3} will yields the estimate that  the operator norm of~$\L'$ is bounded by $C_0\cdot d(B)^{-1}$ and so is the operator norm in (\ref{eqn:q2}), which together with   (\ref{eqn:q1}) completes the proof of  the claim (2).  

To finish the proof, we prove (\ref{eqn:kzz}). 
From hyperbolicity of $B$ and the definition of the function $\mathcal{V}_s(\cdot)$, it is not difficult to check
\[
\int_{\mathcal{E}}\;\frac{\mathcal{V}_s({}^tBw)}{\mathcal{V}_s(z')}\cdot \phi(w-z') \,dw \le C_0\cdot d(B)^{-1}\quad \mbox{for all $z'\in \real^{2d}$}
\] 
and 
\[
\int_{B\mathcal{E}}\;\frac{\mathcal{V}_s(z)}{\mathcal{V}_s({}^tB^{-1}w)}\cdot \phi(z-w) \,dw \le C_0\cdot d(B)^{-1}\quad \mbox{for all $z\in \real^{2d}$.}
\]
By virtue of (\ref{eqn:v}), the former inequality above implies the former claim in (\ref{eqn:kzz}):
\begin{align*}
\int k(z,z')\, dz&\le  C_0\int \left(\int_{\mathcal{E}}\frac{\mathcal{V}_s(B^tw)}{\mathcal{V}_s(z')}\cdot  
\frac{\phi(w-z') }{\langle \|z-{}^tBw\|\rangle^{2d+1}}\,dw\right) dz\le C_0\cdot d(B)^{-1}.
\end{align*}
Similarly the latter inequality above implies the latter claim in (\ref{eqn:kzz}).
\end{proof}

\section{Transfer operators on $\real^{2d+1}$}\label{s:lt}
\subsection{Transfer operators and their lifts}\label{ss:kernelest}
In this section and the following, we consider in the setting of Theorem \ref{th:localmain}: 
Let $F:U\to U'$ be a $\lambda$-hyperbolic contact diffeomorphism with large $\lambda\gg 1$ and $g:\real^{2d+1}\to \complex$ a $C^\infty$ function whose support is contained in $U$; And we consider the transfer operator 
\[
\L=\L_{F,g}:C^{\infty}_0(U')\to C^{\infty}_0(U),\quad \L u(x)= g(x)\cdot u(F(x)).
\]
We define the lift of $\L$ with respect the partial FBI transform $\pFBI$ by
\[
\widehat{\L}:=\pFBI\circ \L\circ \pFBI^*:\Sch(\real^{2d}\oplus \real^{2d+1})\to \Sch(\real^{2d}\oplus \real^{2d+1}).
\]
Then, as (\ref{eqn:cdb}) in the case of FBI transform, the following diagram commutes:
\[
\begin{CD}
\Sch(\real^{2d}\oplus \real^{2d+1})@>{\widehat{\L}}>> \Sch(\real^{2d}\oplus \real^{2d+1})\\
@A{\pFBI}AA @A{\pFBI}AA\\
C^\infty_0(U')@>{\L}>>  C^\infty_0(U)
\end{CD}
\]

The operator $\widehat{\L}$ is an integral operator 
\[
\widehat{\L} u(x_\dag,\xi)=\int K(x_\dag,\xi;z_\dag,\eta)\, u(z_\dag,\eta)\, dz_\dag d\eta
\]
with the kernel
\begin{equation}\label{eqn:K}
K(x_\dag,\xi;z_\dag,\eta)=\int \overline{\Phi_{x_\dag,\xi}(y)}\cdot  g(y)\cdot \Phi_{z_\dag,\eta}(F(y))\, dy.
\end{equation}
Note that $K(x_\dag,\xi;z_\dag,\eta)$ is bounded in absolute value by 
\[
\frac{a_{2d}^2\cdot \langle \xi_0\rangle^{d/2}\cdot  \langle \eta_0\rangle^{d/2}}{2\pi}\cdot 
\left|\int g(y)\cdot e^{- i(\eta F(y)-\xi y)}\cdot e^{-\langle \xi_0\rangle\|x_\dag-y_\dag\|^2/2-\langle \eta_0\rangle\|F_\dag(y_\dag)-z_\dag\|^2/2}\, dy\right|
\]
Applying integration by parts to the integral above, we obtain the following crude estimate on the kernel $K(\cdot)$ of $\widehat{\L}$.
\begin{lemma}
\label{lm:ker}
For any $\rho>0$, there exits a constant $C_\rho>0$, which depends also on $F$ and $g$, such that
\begin{equation}\label{ine:ker}
|K(x_\dag,\xi;z_\dag,\eta)|\le C_\rho \cdot \langle \xi_0\rangle^{d/2}\langle \eta_0\rangle^{d/2}\cdot \int_{\supp\, g}
\kappa(x_\dag,\xi;z_\dag,\eta;y)^{-\rho} dy
\end{equation}
where
\begin{equation}\label{eqn:kappa}
\kappa(x_\dag,\xi;z_\dag,\eta;y)=\langle \xi_0-\eta_0\rangle
\left\langle
\frac{ \|x_\dag-y_\dag\|}{\langle\xi_0\rangle^{-1/2}}\right\rangle
\left\langle
\frac{ \|F_\dag(y_\dag)-z_\dag\|}{\langle \eta_0\rangle^{-1/2}}\right\rangle 
\left\langle 
\frac{\|\xi-{}^tDF_y \eta\|}{ \langle\|\eta\|\rangle^{1/2}} \right\rangle.
\end{equation}
\end{lemma}
\begin{proof}
Consider the differential operators
\[
L_0=\frac{1+i(\eta_0-\xi_0)\cdot \partial_{y_0}}{1+(\eta_0-\xi_0)^2},\qquad  L_1 = \frac{1+i({}^tDF_y(\eta)-\xi)\cdot \partial_{y}}{1+\|{}^{t}DF_y(\eta)-\xi\|^2}.
\]
Since 
\[
L_0\left(e^{-i(\eta F(y)-\xi y)}\right)=L_1\left(e^{-i(\eta F(y)-\xi y)}\right)
=e^{-i(\eta F(y)-\xi y)}, 
\]
we obtain, by integration by parts, that
\begin{align*}
&|K(x_\dag,\xi;z_\dag,\eta)|= \frac{a_{2d}^2\cdot \langle \xi_0\rangle^{d/2}\cdot  \langle \eta_0\rangle^{d/2}}{2\pi}\cdot\left| \int K_{\nu, \nu'}(x_\dag,\xi;z_\dag,\eta;y) dy\right|
\end{align*}
for any integers $\nu, \nu'\ge 0$, where
\[
K_{\nu, \nu'}(x_\dag,\xi;z_\dag,\eta;y)= (L^*_1)^{\nu'}(L_0^*)^{\nu}\left(e^{-\langle \xi_0\rangle \|x_\dag-y_\dag\|^2/2-\langle \eta_0\rangle \|F_\dag(y_\dag)-z_\dag\|^2/2}\cdot  g(y)\right). 
\]
The calculation in the definition of $K_{\nu, \nu'}(\cdot)$ above is not simple, but  one can check the following estimate by an inductive argument on $\nu$ and $\nu'$. (See also Remark \ref{re1} below.): For any integers $\mu,\mu'>0$, there exists a constant $C>0$, which depends also on $F$ and $g$, such that
\begin{equation}\label{eqn:kert}
|K'(x_\dag,\xi;z_\dag,\eta;y)|\le  
C \cdot  T_1(x_\dag,y,\xi_0)^{\mu}\cdot  T_2(y,z_\dag,\eta_0)^{\mu'}
\cdot T_3(\xi_0,\eta_0)^{\nu-\nu'}\cdot T_4(\xi, \eta, y)^{\nu'}
\end{equation}
where $y=(y_0,y_\dag)$ and 
\begin{align*}
&T_1(x_\dag,y,\xi_0)=\frac{1}{ \langle \langle\xi_0\rangle^{1/2}\cdot\|x_\dag-y_\dag\|\rangle},\quad T_2(y,z_\dag,\eta_0)=\frac{1}{\langle \langle\eta_0\rangle^{1/2}\cdot \|F_\dag(y_\dag)-z_\dag\|\rangle},\\
&T_3(\xi_0,\eta_0)=\frac{1}{\langle \xi_0-\eta_0\rangle},\quad T_4(\xi, \eta,y)=
\frac
 { \langle \|\eta\|\rangle^{1/2}}
 {\langle \|{}^tDF_y(\eta)-\xi\|\rangle}+
 \frac
 { \langle \|\eta\|\rangle}
 {\langle \|{}^tDF_y(\eta)-\xi\|\rangle^2}.
\end{align*}
\begin{remark}\label{re1}
In deriving the estimate (\ref{eqn:kert}), use the fact that 
\begin{align*}
&\left| \partial_y^\alpha \left(e^{-\langle \xi_0\rangle \|x_\dag-y_\dag\|^2/2-\langle \eta_0\rangle \|F_\dag(y_\dag)-z_\dag\|^2/2}\cdot  g(y)\right)\right|\\
&\qquad \qquad \le C_{\alpha,\mu,\mu'}\cdot  \max\{\langle \xi_0\rangle, \langle \eta_0\rangle\}^{|\alpha|/2}\cdot T_1(x_\dag,y,\xi_0)^\mu\cdot T_2(y,z_\dag,\eta_0)^{\mu'}
\end{align*}
for any multi-index $\alpha$ and any $\mu, \mu'\ge 0$, and that $
\langle \xi_0\rangle\le \langle \eta_0\rangle +\langle \xi_0-\eta_0\rangle$.
\end{remark}

In the case where $\langle\|\eta\|\rangle^{1/2}\le \langle\| {}^tDF_y(\eta)-\xi\|\rangle$ holds, we have
\[
T_4(\xi, \eta,y)\le \frac{4}{\langle \| {}^tDF_y(\eta)-\xi\|/\langle \|\eta\|\rangle^{1/2}\rangle} 
\]
and, hence, the inequality (\ref{ine:ker}) follows from (\ref{eqn:kert}). 
Otherwise, the inequality  (\ref{ine:ker}) follows again from the same estimate  (\ref{eqn:kert}) but with setting $\nu'=0$.
\end{proof}

\subsection{Decomposition of transfer operators}
Our next task is to decompose the transfer operator $\L$ and its lift $\widehat{\L}$ into three parts, namely, {\em the compact, hyperbolic and central part}. 
The decomposition depends on two constants. One of the constants is $\tau>1/2$. 
We take $\tau$ close to $1/2$ according to $d$ and $r$. For instance, it is quite enough to assume
\[
\frac{1}{2}<\tau<\frac{1}{2}+\frac{1}{100\cdot  r\cdot  d}.
\]
The other constant is $N>0$. We will choose $N$ as a large constant in the course of the argument below so that  several claims hold true. Note that the choice of $N$ will depend on $F$, $g$ and $\tau$.
 
We define two function
\[
X_0:\real^{2d}\oplus \real^{2d+1}\to [0,1] \quad\mbox{and}\quad  X_{\ctr}:\real^{2d}\oplus \real^{2d+1}\to [0,1]
\]
as follows: Recall the function $\chi$ introduced in the beginning of Section \ref{s:pFBI}. In the case where $x=0$, we set 
\[
X_0(0,\xi)=\chi(\|\xi\|/N)
\qquad\mbox{and}\qquad
 X_{\ctr}(0,\xi)=\chi\left(\frac{\|\xi_\dag\|}{\langle\xi_0\rangle^{\tau}}\right)
\]
for $\xi=(\xi_0, \xi_\dag)\in \real^{2d+1}$. Then we extend these definitions to the case $x\neq 0$ uniquely so that they are invariant with respect to the natural action of the transformation group $\Af$ on $\real^{2d}\oplus \real^{2d+1}$. More concretely, we set
\begin{align*}
&X_0(x_\dag,\xi)=\chi((|\xi_0|^2+\|\xi-\xi_0\cdot \alpha_0(x_\dag)\|^2)^{1/2}/N)
\intertext{and}
&X_{\ctr}(x_\dag,\xi)=\chi\left(\frac{\|\xi-\xi_0\cdot \alpha_0(x_\dag)\|}{\langle\xi_0\rangle^{\tau}}\right). 
\end{align*}
Note that the support of $X_{0}$ is contained in the neighborhood 
\[
\left\{(x_\dag,\xi)\in \real^{2d}\oplus\real^{2d+1}\;\left| \; |\xi_0|^2+\|\xi-\xi_0\cdot \alpha_0(x_\dag)\|^2\le (2N)^2\right.\right\}
\]
of the zero section, while that of $X_{\ctr}$ is contained in the neighborhood 
\[
\left\{(x_\dag,\xi)\in \real^{2d}\oplus\real^{2d+1}\;\left| \; \|\xi-\xi_0\cdot \alpha_0(x_\dag)\| \le 2\cdot \langle \xi_0\rangle^{\tau}\right.\right\}
\]
of the one-dimensional subbundle spanned by $\alpha_0$.

We define the compact, central and hyperbolic part of the lift $\widehat{\L}$ respectively by
\begin{alignat*}{2}
&\widehat{\L}_{\cpt}:\Sch(\real^{2d}\oplus \real^{2d+1})\to \Sch(\real^{2d}\oplus \real^{2d+1}),&\quad 
&\widehat{\L}_{\cpt}u= \widehat{\L}(X_0\cdot u),\\
&\widehat{\L}_{\ctr}:\Sch(\real^{2d}\oplus \real^{2d+1})\to \Sch(\real^{2d}\oplus \real^{2d+1}),&\quad 
&\widehat{\L}_{\ctr} u= \widehat{\L}(X_{\ctr}\cdot(1-X_0)\cdot  u),\quad \mbox{and}\\
&\widehat{\L}_{\hyp}:\Sch(\real^{2d}\oplus \real^{2d+1})\to \Sch(\real^{2d}\oplus \real^{2d+1}),&\quad 
&\widehat{\L}_{\hyp} u= \widehat{\L}((1-X_{\ctr})\cdot (1-X_0)\cdot u).
\end{alignat*}
Clearly the lift $\widehat{\L}$ is decomposed into these three operators:
\[
\widehat{\L}=\widehat{\L}_{\cpt}+\widehat{\L}_{\ctr}+\widehat{\L}_{\hyp}.
\]
For the transfer operator $\L$ itself, we define its compact, central and hyperbolic part respectively by 
\[
\L_{\sigma}:\Sch(\real^{2d+1})\to \Sch(\real^{2d+1}), \quad \L_{\sigma}=
\pFBI^*\circ \widehat{\L}_{\sigma} \circ \pFBI
\]
with $\sigma=\cpt, \ctr, \hyp$, so that  we have $\L=\L_{\cpt}+\L_{\ctr}+\L_{\hyp}$.

\subsection{The compact part}
The compact part of the transfer operator $\L$ is in fact a compact operator, as the name indicates. 
\begin{lemma}\label{lm:cpt}
The compact part $\L_{\mathrm{cpt}}$ extends to a bounded operator $
\L_{\mathrm{cpt}}:H^r_{\aniso}\to H^r_{\aniso}$. 
This extension is a trace class operator and hence a compact operator. 
\end{lemma}
\begin{proof}
We can check that the operator $\L_{\mathrm{cpt}}$ is an integral operator with smooth kernel and maps the Sobolev space $H^r$ into $C^\infty(\supp\, g)$ continuously. 
Thus one may view  the operator $\L_{\mathrm{cpt}}:H^r_{\aniso}\to H^r_{\aniso}$ as the composition
\[
\begin{CD}
H^r_{\aniso}@>{\L_{\mathrm{cpt}}}>>  H^s(\supp\, g)@>{\iota}>> H^r(\supp\, g)@>{\iota}>> H^r_{\aniso}
\end{CD}
\]
where $s>r$ is an arbitrarily large number and $\iota$ denotes the injections. 
Since the injection $\iota:H^s(\supp\, g)\to H^r(\supp\, g)$ is a trace class operator if $s-r$ is sufficiently large\cite[Ch.10.2]{YoshidaBook} and since the composition of a trace class operator with a bounded operator is again a trace class operator, we obtain the lemma.
\end{proof}
The hyperbolic and central part will be considered in the following two sections.

\section{The hyperbolic part}\label{s:hyp}
In this section, we consider the hyperbolic part of the transfer operator. We will use the notation in the previous sections and set $X_\hyp=(1-X_0)\cdot (1-X_\ctr)$ for simplicity. 
From  Lemma \ref{lm:ker}, we see that the action of $\widehat{\L}$ is closely related to the pull-back operator by the natural action 
\[
\widetilde{F}:\real^{2d}\oplus \real^{2d+1}\to \real^{2d}\oplus \real^{2d+1},\quad \widetilde{F}(x_\dag,\xi)=(F_\dag(x_\dag),  {}^{t}DF_{(0,x_\dag)}^{-1}(\xi))
\]
of $F$ on $\real^{2d}\oplus \real^{2d+1}$ post-composed by the multiplication by $g(x_\dag)$. 
And, by the definition of the weight function $\weight^r_\aniso$ and hyperbolicity of $F$, we have 
\[
\weight^r_\aniso(\widetilde{F}^{-1}(x_\dag,\xi))\le C_0 \cdot \lambda^{-r}\cdot  \weight^r_\aniso(x_\dag,\xi)
\qquad \mbox{for $(x_\dag,\xi)\in \supp X_\hyp$}
\]
where $C_0>0$ is an absolute constant. (In fact, the weight function $\weight^r_\aniso$ is designed so that this inequality holds.) 
In view of these observations, the claim of the next proposition  should be a natural one. 
\begin{proposition}
\label{pp:hyp}
The hyperbolic part $\L_{\mathrm{hyp}}$ extends naturally to the bounded operator $\L_{\mathrm{hyp}}:H^r_{\aniso}\to H^r_{\aniso}$. 
Further the operator norm of the extension is bounded by $C_0\cdot \|g\|_\infty\cdot \lambda^{-r}$, where $C_0>0$ is a constant independent of $F$ and $g$. 
\end{proposition}
We give an elementary proof for this proposition in the following subsections. 
But, since the proposition is intuitively rather obvious as we observed above and since this is not a main point of our argument, one may skip the proof below and proceed to the next section where we treat the central part.

\subsection{A Littlewood-Paley type partition of unity}
To begin with, we  introduce a partition of unity on $\real^{2d}\oplus \real^{2d+1}$  and then define a norm on $H^r_\aniso$, which is equivalent to the original norm $\|\cdot\|_r$ but more tractable for our purpose. 

First we consider  a simple partition of unity $\{\chi_n:\real\to [0,1]\mid  n=0,1,2,\dots\}$ on the real line $\real$  
defined by
\[
\chi_n(s)=\begin{cases}
\chi(|s|),&\mbox{ if $n=0$;}\\
\chi(2^{-n}|s|)-\chi(2^{-n+1}|s|),&\mbox{ if $n\ge 1$,}
\end{cases}
\]
where $\chi$ is the function  introduced in the beginning of Section \ref{s:pFBI}.
Using this partition of unity and recalling the functions $\psi_{\pm}:\mathbf{P}\real^{2d}\to [0,1]$ used in the definition of~$\weight^r_\aniso$, we define the $C^\infty$ partition of unity   
 $\{\psi_m:\real^{2d}\to [0,1]\}_{m\in \integer}$ on $\real^{2d}$ by
\[
\psi_m(\xi_\dag)=
\begin{cases}
\chi_{m}(\|\xi_\dag\|)\cdot \psi_+([\xi_\dag]), &\mbox{if $m>0$;}\\
\chi_0(\|\xi_\dag\|), &\mbox{if $m=0$;}\\
\chi_{|m|}(\|\xi_\dag\|)\cdot \psi_-([\xi_\dag]), &\mbox{if $m<0$.}
\end{cases}
\]
By this definition, there exists a constant $C_0>1$ such that 
\[
 C_0^{-1}\sum_{m\in \integer} 2^{-2rm} \cdot {\psi}_{m}(\xi_\dag)^2\le 
 W^r_{\aniso}(\xi_\dag)^{2}
  \le C_0\sum_{m\in \integer} 2^{-2rm} \cdot {\psi}_{m}(\xi_\dag)^2\quad \mbox{ for $\xi_\dag\in \real^{2d}$.}
\]
Finally we define the $C^\infty$ partition of unity $\{\Psi_m:\real^{2d}\oplus \real^{2d+1}\to [0,1]\}_{m\in \integer}$ on $\real^{2d}\oplus \real^{2d+1}$
as follows. In the case $x_\dag=0$, we set
\[
\Psi_{m}(0, \xi)=
 \psi_{m}\left(
 \frac{\|\xi_\dag\|}{\langle \|\xi\|\rangle^{1/2}}\right)\quad \mbox{where $\xi=(\xi_0, \xi_\dag)$,}
 \]
 and then extend this definition to the case $x\neq 0$ uniquely so that it is invariant with respect to the natural action of the transformation group $\Af$. That is to say, we set
\[
\Psi_{m}(x_\dag, \xi)=
 \psi_{m}\left(
 \frac{\|\xi-\xi_0 \cdot \alpha_0(x_\dag)\|}{\langle (|\xi_0|^2+\|\xi-\xi_0 \cdot \alpha_0(x_\dag)\|^2)^{1/2}\rangle^{1/2}}\right).
 \]
The estimate on the function $W^r_\aniso$ above implies that 
\[
 C_0^{-1}\cdot \sum_{m\in \integer} 2^{-4rm} \cdot {\Psi}_{m}(x_\dag,\xi)^2\le 
 \weight^r_{\aniso}(x_\dag,\xi)^{2}
 \le C_0\cdot \sum_{m\in \integer} 2^{-4rm} \cdot {\Psi}_{m}(x_\dag,\xi)^2
\]
for $(x_\dag,\xi)\in \real^{2d}\oplus \real^{2d+1}$. Consequently the anisotropic Sobolev norm $\|\cdot\|_r$ satisfies
\[
C_0^{-1}\cdot { \sum_{m\in\integer} 2^{-4rm}\cdot  \| \Psi_{m}\cdot \pFBI u\|_{L^2}^2}\le 
{\|u\|_r^2}\le C_0\cdot  { \sum_{m\in\integer} 2^{-4rm}\cdot  \| \Psi_{m}\cdot \pFBI u\|_{L^2}^2}.
\]
Therefore, in order to prove Proposition \ref{pp:hyp}, it is enough to prove the same claim with the norm $\|\cdot\|_r$ replaced by the new norm
\[
\|u\|'_r=\left({ \sum_{m\in\integer} 2^{-4rm}\cdot  \| \Psi_{m}\cdot \pFBI u\|_{L^2}^2}\right)^{1/2}.
\]

\subsection{The proof of Proposition \ref{pp:hyp}}
Below we give the proof of Proposition \ref{pp:hyp}  assuming a lemma (Lemma \ref{lm:sep}) whose proof is postponed until the next subsection. 
Take a function $u\in \Sch(\real^{2d+1})$ arbitrarily and set
\[
v=\L_\hyp u,\qquad u_{m}=\Psi_{m}\cdot \pFBI u, \qquad v_{m}=\Psi_{m}\cdot \pFBI v=
\Psi_{m}\cdot \widehat{\L}_\hyp\circ  \pFBI u.
\]
As we noted at the end of the last subsection, it is enough to show the claim
\begin{equation}\label{eqn:cka}
\sum_{m\in \integer} 2^{-4rm}\cdot  \| v_{m}\|_{L^2}^2
\le C_0 \cdot \|g\|_{\infty}^2\cdot \lambda^{-2r}\cdot \sum_{m\in \integer} 2^{-4rm}\cdot  \| u_{m}\|_{L^2}^2.
\end{equation}
To proceed, we decompose $v_{m}$ into countably many pieces
\[
v_{m,m'}=\Psi_{m}\cdot \widehat{\L}_\hyp(u_{m'}),\qquad m'\in \integer.
\]
The claim (\ref{eqn:cka}) is a consequence of the following three facts on the relation between the $L^2$ norms of $v_{m,m'}$ and $u_{m'}$. The first is the fact that
\[
\|v_{m,m'}\|_{L^2}\le \|g\|_{L^\infty}\cdot \|u_{m'}\|_{L^2}
\]
for any $m,m'\in \integer$. 
This follows from the facts that the operators $\pFBI$ and $\pFBI^*$ do not increase the $L^2$ norm and that
the operator norm of $\L$ with respect to the $L^2$ norm is bounded by $\|g\|_{L^\infty}$. The second is that, for any given $N'>0$, we can take the constant $N>0$ in the definition of $X_0$ so that $v_{m,m'}=0$ whenever $|m'|\le N'$. 
In fact, if we take sufficiently large $N>0$ according to $N'$, 
the condition $|m'|\le N'$ implies that $\supp X_{\hyp}\cap \supp \Psi_{m'}=\emptyset$ and hence
that  $\widehat{\L}_\hyp(u_{m'})=0$. 
The third is stated in the following lemma:
\begin{lemma}\label{lm:sep}
For any $\nu>0$, there exists a constant $C_\nu>0$, which may depend on $F$ and $g$, such that, if 
\begin{equation}\label{eqn:nc}
m< m'+(1/2)\log_2 \lambda-3
\end{equation}
then 
\[
\| v_{m,m'}\|_{L^2}\le C_\nu \cdot 2^{-\nu\cdot \max\{m,m'\} }\|u_{m'}\|_{L^2}.
\]
\end{lemma}
As we will see later, the assumption (\ref{eqn:nc})  is a sufficient condition for the image of 
$\supp\, \Psi_{m'}$ by $\widetilde{F}^{-1}$ not to intersect $\supp\, \Psi_m$ and therefore the claim of Lemma \ref{lm:sep} is quite natural. 
But, unfortunately, our proof of Lemma \ref{lm:sep} is not very short. So we will give it in the next subsection and below we finish the proof of Proposition~\ref{pp:hyp} assuming Lemma \ref{lm:sep}.

Take and fix $\nu$ such that $\nu>2r$. Also take large $N'>0$ and set
\[
\gamma_{m,m'}=\begin{cases}
0&\mbox{if $|m'|\le N'$;}\\
2^{2r(m'-m)}\cdot \|g\|_\infty&\mbox{if $|m'|\ge N'$ and $m\ge   m'+(1/2)\log_2\lambda-3$;}\\
C_\nu\cdot  2^{-(\nu-4r)\cdot \max\{m,m'\}}&\mbox{if $|m'|\ge N'$ and $m< m'+(1/2)\log_2\lambda-3$.}
\end{cases}
\]
We may and do take large $N'>0$ so that 
\begin{equation}\label{eqn:su}
\sum_{m\in \integer}\gamma_{m,m'}<C_0 \cdot \|g\|_\infty \cdot \lambda^{-r},\qquad 
\sum_{m'\in \integer}\gamma_{m,m'}<C_0 \cdot \|g\|_\infty \cdot \lambda^{-r}
\end{equation}
where $C_0$ is a constant independent of $F$ and $g$. (Take large $N'>0$ according to $F$ and $g$ so that $C_\nu\cdot 2^{-(\nu-r) N'}$ is small.)
From Lemma \ref{lm:sep} and the two facts stated in the paragraph preceding it, we have also
\begin{equation}\label{eqn:rm}
2^{-2rm}\|v_{m,m'}\|_{L^2}\le \gamma_{m,m'}\cdot 2^{-2rm'} \|u_{m'}\|_{L^2}\quad 
\mbox{for any $(m,m')\in \integer\oplus \integer$,}
\end{equation}
provided that we take sufficiently large $N$ according to $N'$. 
Now,  by using  Schwarz lemma, we obtain the required estimate (\ref{eqn:cka}):
\begin{align*}
\sum_{m\in\integer} 2^{-4rm}\cdot  \| v_{m}\|_{L^2}^2&=
\sum_{m\in\integer} 2^{-4rm}\cdot  \left\| \sum_{m'\in\integer}  v_{m,m'}\right\|_{L^2}^2\\
&\le 
\sum_{m\in\integer}\left( 2^{-4rm}\cdot  \left(\sum_{m'\in\integer} \gamma_{m,m'}\right)\cdot \left( \sum_{m'\in\integer} \gamma_{m,m'}^{-1} \left\|  v_{m,m'}\right\|_{L^2}^2\right)\right)\\
&\le \left(C_0 \|g\|_\infty \cdot \lambda^{-r}\right)\cdot \sum_{m\in\integer}\sum_{m'\in\integer}
2^{-4rm}\gamma_{m,m'}^{-1}\cdot \|v_{m,m'}\|_{L^2}^2\quad \mbox{by (\ref{eqn:su})}\\
&\le \left(C_0  \|g\|_\infty \cdot \lambda^{-r}\right)\cdot \sum_{m\in\integer}\sum_{m'\in\integer}
2^{-4rm'}\gamma_{m,m'}\cdot \|u_{m'}\|_{L^2}^2\quad \mbox{by (\ref{eqn:rm})}\\
&\le \left(C_0 \|g\|_\infty \cdot \lambda^{-r}\right)^2\cdot \sum_{m'\in\integer}
2^{-4rm'} \|u_{m'}\|_{L^2}^2\quad \mbox{by (\ref{eqn:su}).}
\end{align*}
This completes the proof of Proposition \ref{pp:hyp}.

\subsection{Consequences of the condition (\ref{eqn:nc})}
Before we go into the proof of Lemma \ref{lm:sep}, we give consequences of the assumption (\ref{eqn:nc}) in the lemma. 
\begin{lemma}\label{lm:le}
There exist a small constant $c>0$ and a large constant $C>0$, which depend on $F$, such that,
if $m$ and $m'$ satisfy  (\ref{eqn:nc}) and $|m'|>C$, and if
\[
y=(y_0,y_\dag)\in \supp\, g,\quad  (y_\dag, \xi')\in \supp\, \Psi_{m},\quad\mbox{and}\quad (F_\dag(y_\dag), \eta')\in \supp\, \Psi_{m'},
\]
it holds
\begin{equation}\label{eqn:le}
\|\xi'-{}^tDF_{y}(\eta')\|\ge c\cdot 2^{\max\{|m|, |m'|\}}\cdot\max\{\langle \|\xi'\|\rangle^{1/2}, \langle \|\eta'\|\rangle^{1/2}\}.
\end{equation}
\end{lemma}

\begin{proof}
Since the function $ \Psi_{m}$ is invariant with respect to the natural action of $\Af$, 
we may and do assume $y=F(y)=0$, changing coordinates by affine transformations in $\Af$.
Put $\tilde{\eta}={}^tDF_0(\eta')$ and write $\xi'$, $\eta'$ and $\tilde{\eta}$ as
\[
\xi'=(\xi'_0, \xi'_\dag), \quad \eta'=(\eta'_0,\eta'_\dag), \quad \tilde{\eta}=(\tilde{\eta}_0, \tilde{\eta}_\dag) \in \real\oplus \real^{2d}
\]
Recall that $F$ is written in the form (\ref{Fh}) in a neighborhood of $0$ and the function $f$ in it satisfies 
(\ref{eqn:basic}). In particular,  
we have $\tilde{\eta}_\dag={}^t(DF_\dag)_0\eta'_\dag$ and $\tilde{\eta}_0=\eta_0$. 
From the assumption $(y_\dag, \xi')\in \supp\, \Psi_{m}$ and $(F_\dag(y_\dag), \eta')\in \supp\, \Psi_{m'}$, we have
\begin{equation}\label{lmass}
\frac{\xi'_\dag}{\langle \|\xi'\|\rangle^{1/2}}\in \supp\,  \psi_m, \qquad \frac{\eta'_\dag}{\langle \|\eta'\|\rangle^{1/2}}\in \supp\, \psi_{m'}.
\end{equation}

To proceed, we consider the position of the point $\tilde{\eta}_\dag/\langle \tilde{\eta}\rangle^{1/2}$ in $\real^{2d}$. 
We  consider the cases $m'<0$ and $m'>0$ separately. 
In the case $m'<0$, we  claim
\[
\frac{\tilde{\eta}_\dag}{\langle \|\tilde{\eta}\|\rangle^{1/2}}\in \{ z\mid \|z\|\le \lambda^{-1/2} \cdot 2^{m'+1}\}\cup \cone^*_+(1/10).
\]
To prove this claim, we suppose $\tilde{\eta}_\dag \in \real^{2d}\setminus \cone^*_+(1/10)$ and show
\[
\frac{\|\tilde{\eta}_\dag\|}{\langle \|\tilde{\eta}\|\rangle^{1/2}}\le \lambda^{-1/2}\cdot 2^{m'+1}.
\]
By $\lambda$-hyperbolicity of $F$, we have $
\|\tilde{\eta}_\dag\|\le \lambda^{-1}\|\eta'_\dag\|$. 
Since $\tilde{\eta}_0=\eta'_0$, we also have
\[
\frac{\|\tilde{\eta}_\dag\|}{\|\eta'_\dag\|
}\le \frac{\langle\|\tilde{\eta}\|\rangle}{\langle \|\eta'\|\rangle}\le 1.
\]
Combining these two inequalities, we get
\[
\frac{\|\tilde{\eta}_\dag\|}{\langle \|\tilde{\eta}\|\rangle^{1/2}}=
\|\tilde{\eta}_\dag\|^{1/2}\cdot \frac{\|\tilde{\eta}_\dag\|^{1/2}}{\langle \|\tilde{\eta}\|\rangle^{1/2}}
\le 
\lambda^{-1/2}
\|\eta'_\dag\|^{1/2}\cdot \frac{\|\eta'_\dag\|^{1/2}}{\langle \|\eta'\|\rangle^{1/2}}
\le 
\lambda^{-1/2}\cdot  \frac{\|\eta'_\dag\|}{\langle \|\eta'\|\rangle^{1/2}}.
\]
This together with the latter condition in (\ref{lmass}) implies the required estimate. 
In the case $m'>0$, a similar argument yields 
\[
\frac{\tilde{\eta}_\dag}{\langle \|\tilde{\eta}\|\rangle^{1/2}}\in \{ z\mid \|z\|\ge \lambda^{1/2} \cdot 2^{m'-1}\}\cap \cone^*_+(1/10).
\]

Compare the claims on the position of $\tilde{\eta}_\dag/\langle \tilde{\eta}\rangle$ with 
the first condition in (\ref{lmass}) and recall the assumption (\ref{eqn:nc}). Then we see that there exists a small constant $c>0$ and a large constant $C>0$, which may depend on $F$, such that, if $m$ and $m'$ satisfy (\ref{eqn:nc}) and $|m'|>C$, we have
\[
\left|\frac{\xi'_\dag}{\langle \|\xi'\|\rangle^{1/2}}-\mu \cdot \frac{\tilde{\eta}_\dag}{\langle \|\tilde{\eta}\|\rangle^{1/2}}\right|\ge c \cdot 2^{\max\{|m|, |m'|\}}\quad \mbox{for any $1/2\le \mu\le 2$. }
\]

If  $1/2\le \langle \|\xi'\|\rangle/\langle \|\tilde{\eta}\|\rangle \le 2$,  the last estimate implies
\[
\|\xi'-\tilde{\eta}\|\ge \|\xi'_\dag-\tilde{\eta}_\dag\|\ge (c/2) \cdot 2^{\max\{|m|, |m'|\}}\cdot \max\{\langle \|\xi'\|\rangle^{1/2},\langle \|\tilde{\eta}\|\rangle^{1/2}\}
\]
and  hence the inequality (\ref{eqn:le}) holds with a possibly different constant $c>0$.  (Note that the ratio between $\|\tilde{\eta}\|$ and $\|\eta'\|$ is bounded by a constant that depends on $F$.)

In the remaining case where either $ \langle \|\xi'\|\rangle/\langle\| \tilde{\eta}\|\rangle<1/2$ or $\langle \|\xi'\|\rangle/\langle\| \tilde{\eta}\|\rangle > 2$ holds,  we can prove the inequality (\ref{eqn:le}) by a crude estimate as follows. Clearly we have
\[
\|\xi'-\tilde{\eta}\|\ge |\langle \|\xi'\|\rangle-\langle\|\tilde{\eta}\|\rangle |\ge  \frac{1}{2} \cdot \max\{ \langle \|\xi'\|\rangle, \langle\|\tilde{\eta}\|\rangle\}
\]
in this case. Since (\ref{lmass}) implies
\[
\langle \|\xi'\|\rangle^{1/2}\ge \frac{\|\xi'_\dag\|}{\langle \|\xi'\|\rangle^{1/2}}\ge  2^{|m|-1},\qquad 
\langle \|\eta'\|\rangle^{1/2}\ge \frac{\|{\eta'}_\dag\|}{\langle \|\eta'\|\rangle^{1/2}}\ge  2^{|m'|-1},
\]
we have also
\[
\max\{ \langle \|\xi'\|\rangle^{1/2}, \langle\|\tilde{\eta}\|\rangle^{1/2}\}\ge \frac{1}{2}\cdot 2^{\max\{|m|, |m'|\}}.
\]
Therefore the inequality (\ref{eqn:le}) holds for a sufficiently small constant $c$. 
\end{proof}
\begin{corollary}\label{cor:k}
There exist a small constant $c>0$ and a large constant $C>0$, which depend on $F$, such that,
if $m$ and $m'$ satisfy  (\ref{eqn:nc}) and $|m'|>C$, we have
\[
\kappa(x_\dag,\xi;z_\dag,\eta;y)\ge c\cdot 2^{\max\{|m|,|m'|\}}
\]
 for any $(x_\dag,\xi)\in \supp \Psi_m$, $(z_\dag,\eta)\in \supp ( \Psi_{m'}\cdot  X_\hyp)$ and $y\in \supp\, g$.
\end{corollary}
\begin{proof}
From the assumptions $(z_\dag,\eta)\in \supp\, \Psi_{m'}$ and $(x_\dag,\xi)\in \supp\, \Psi_m$ and from invariance of $\Psi_m$ with respect to the natural action of the transformation group $\Af$, we have
\begin{align*}
&(y_\dag, {}^tDA_{(0,x_\dag-y_\dag)}(\xi))=(y_\dag, \xi-\xi_0(\alpha_0(x_\dag)-\alpha_0(y_\dag)))\in \supp\, \Psi_{m}\quad\mbox{and}\\
&(F_\dag(y_\dag), {}^tDA_{(0,z_\dag-F_\dag(y_\dag))}(\eta))=(F_\dag(y_\dag), \eta-\eta_0(\alpha_0(z_\dag)-\alpha_0(F_\dag(y_\dag))))\in \supp\,  \Psi_{m'}.
\end{align*}
Hence we can apply  Lemma \ref{lm:le} to the setting 
\[
\xi'=\xi-\xi_0(\alpha_0(x_\dag)-\alpha_0(y_\dag)),\qquad \eta'=\eta-\eta_0(\alpha_0(z_\dag)-\alpha_0(F_\dag(y_\dag))),
\]
and obtain the estimate (\ref{eqn:le}) as the conclusion. Note that the estimate (\ref{eqn:le}) implies in particular that there exits a constant $c'>0$, which depends on $F$, such that 
\begin{equation}\label{eqn:we}
\max\{\langle \|\xi'\|\rangle, \langle \|\eta'\|\rangle \}\ge c'\cdot 2^{\max\{|m|, |m'|\}}\cdot\max\{\langle \|\xi'\|\rangle^{1/2}, \langle \|\eta'\|\rangle^{1/2}\}.
\end{equation}
Since we have
\[
\|\xi-\xi'\|=|\xi_0|\cdot \|\alpha_0(x_\dag)-\alpha_0(y_\dag)\|=|\xi_0|\cdot \|x_\dag-y_\dag\|
\]
and
\[
\|\eta-\eta'\|=|\eta_0|\cdot \|\alpha_0(F_\dag(y_\dag))-\alpha_0(z_\dag)\|=|\eta_0|\cdot \|F_\dag(y_\dag)-z_\dag\|,
\]
the estimate (\ref{eqn:le}) implies also that we can take a small constant $c''>0$, which depend on $F$,   so that  either of the following three inequalities holds: 
\begin{itemize}
\setlength{\itemsep}{6pt plus 1pt minus 1pt}
\item[(a)]$\displaystyle \|\xi-\xi'\|=|\xi_0|\cdot \|x_\dag-y_\dag\| \ge c'' \cdot 2^{\max\{|m|,|m'|\}}\cdot \max\{\langle \|\xi'\|\rangle^{1/2}, \langle \|\eta'\|\rangle^{1/2}\}$, 
\item [(b)]
$\displaystyle \|\eta-\eta'\|=|\eta_0|\cdot \|F_\dag(y_\dag)-z_\dag\|\ge c'' \cdot 2^{\max\{|m|,|m'|\}}\cdot \max\{\langle \|\xi'\|\rangle^{1/2}, \langle \|\eta'\|\rangle^{1/2}\}$,
\item[(c)] $\displaystyle \|{}^tDF_y(\eta)- \xi\|\ge c''\cdot 2^{\max\{|m|,|m'|\}}\cdot \max\{\langle \|\xi'\|^{1/2}\rangle, \langle \|\eta'\|^{1/2}\rangle\}$.
\end{itemize}
Further, from (\ref{eqn:we}), we may take the constant $c''>0$ above so small that,  if neither of the inequality (a) nor (b) above holds, we have
\[
\max\{\|\xi'\|, \|\eta'\|\}\ge\frac{1}{2} \max\{\langle \|\xi\|\rangle, \langle \|\eta\|\rangle\}.
\]
Hence, with such choice of the constant $c''>0$, we always have either (a), (b) or 
\begin{itemize}
\item[(d)] $\displaystyle \|{}^tDF_0(\eta)- \xi\|\ge (c''/2)\cdot 2^{\max\{|m|,|m'|\}}\cdot \max\{\langle \|\xi\|\rangle^{1/2}, \langle \|\eta\|\rangle^{1/2}\}$.
\end{itemize}
Clearly the inequalities (a), (b) and  (d) imply respectively that the second, third and fourth term in the definition (\ref{eqn:kappa}) of $\kappa(\cdot)$ is so large that the conclusion of the corollary  holds for a sufficiently small constant $c>0$.
\end{proof}

\subsection{Proof of Lemma \ref{lm:sep}}\label{ss:pf7} 
From Lemma \ref{lm:ker} and Corollary \ref{cor:k}, for arbitrarily large $\mu>0$, there exists a constant $C_{\mu}>0$, which may depend on $F$ and $g$ but not on $m$ nor $m'$, such that 
\[
|v_{m,m'}(x_\dag,\xi)|
\le C_{\mu}\cdot  2^{-\mu \cdot \max\{|m|, |m'|\}}  \int_{\supp \Psi_{m'}\cap \supp\, X_{\hyp}} {\!\!\!\!\!\!\!\!\!\!} K_\mu(x_\dag,\xi;z_\dag,\eta)  |u_{m'}(z_\dag,\eta)| dz_\dag d\eta  
\]
where 
\[
K_\mu(x_\dag,\xi;z_\dag,\eta)=
\frac{\langle \xi_0\rangle^{d/2}\langle \eta_0\rangle^{d/2}}{
\langle |\xi_0-\eta_0|\rangle^{\mu}} \int_{\supp g} 
\left\langle
\frac{ \|x_\dag-y_\dag\|}{\langle\xi_0\rangle^{-1/2}}\right\rangle^{-\mu} 
\left\langle
\frac{ \|F_\dag(y_\dag)-z_\dag\|}{\langle \eta_0\rangle^{-1/2}}\right\rangle^{-\mu} dy.
\]
Note also that  the support of $v_{m,m'}$ is contained in $\supp\, \Psi_{m}$.
Hence, by the Schur test, the conclusion of Lemma~\ref{lm:sep} follows if we show
\begin{claim}There exists a constant $\nu>0$, which depends only on $d$ and $\tau$, such that 
\[
\int_{\supp \Psi_{m}} K_\mu (x_\dag,\xi;z_\dag,\eta) dx_\dag d\xi<C \cdot 2^{\nu\cdot \max\{|m|, |m'|\}}
\]
for $(z_\dag,\eta)\in \supp \Psi_{m'}\cap\, \supp\, X_{\hyp}$
and
\[
\int_{\supp \Psi_{m'}\cap\, \supp\, X_{\hyp}} K_\mu (x_\dag,\xi;z_\dag,\eta) dz_\dag d\eta<C \cdot 2^{\nu\cdot \max\{|m|, |m'|\}} 
\]
for $(x_\dag,\xi)\in \supp \Psi_{m}$, where $C>0$ is a constant that does not depend on $m$ nor $m'$. 
\end{claim}

For $(x_\dag,\xi)\in \supp \Psi_m$, we have that 
\[
\frac{\|\xi_\dag-\xi_0\cdot \alpha_0(x_\dag)\|}{\langle 2\cdot\max\{|\xi_0|, \|\xi_\dag-\xi_0\cdot \alpha_0(x_\dag)\|\}\rangle^{1/2}}\le
\frac{\|\xi_\dag-\xi_0\cdot \alpha_0(x_\dag)\|}{\langle (\xi_0^2+\|\xi_\dag-\xi_0\cdot \alpha_0(x_\dag)\|^2)^{1/2}\rangle^{1/2}}\le 2^{m+1}
\]
and hence in particular that
\[
\|\xi_\dag-\xi_0\cdot \alpha_0(x_\dag)\|\le  2^{2|m|+4}\cdot \langle \xi_0\rangle^{1/2}.
\]
Similarly,  for $(z_\dag,\eta)\in  \supp \Psi_m$, we have
\[
\|\eta_\dag-\eta_0\cdot \alpha_0(z_\dag)\|\le  2^{2|m'|+4}\cdot \langle \eta_0\rangle^{1/2}.
\]
On the other hand, if  $(z_\dag,\eta)\in \supp\, X_\hyp$, we have
\[
\langle \eta_0\rangle^{\tau}\le \|\eta_\dag-\eta_0\cdot \alpha_0(z)\|
\]
by definition. Hence, for $(z_\dag,\eta)\in \supp \Psi_{m'}\cap\, \supp\, X_{\hyp}$, we have
\[
| \eta_0|\le  2^{(|m'|+4)/(\tau-(1/2))}.
\]
We can show the claim above just by calculating the integrals in the statement using these estimates.


\section{The central part}\label{s:ctr}
In this section, we deal with the central part of the transfer operator. 
As in the last section, we consider the situation assumed in Theorem \ref{th:localmain} and use the notation prepared in the previous sections. 
We will prove 
\begin{proposition}\label{pp:ctrnorm}
The central part $\L_{\mathrm{ctr}}$ of the transfer operator $\L$ extends naturally  to a bounded linear operator 
$\L_{\mathrm{ctr}}:H^r_{\mathrm{aniso}}\to H^r_{\mathrm{aniso}}$. Further, there exists a constant $C_0>0$, which does not depend on $F$ and $g$, such that, if we take sufficiently large number for the constant  $N$ according to $F$ and $g$,  
the operator norm of the extension $\L_{\mathrm{ctr}}:H^r_{\mathrm{aniso}}\to H^r_{\mathrm{aniso}}$ is bounded by
\begin{equation}\label{eqn:cst}
C_0\cdot\max\{\Lambda(F,g),  \|g\|_{\infty}\cdot \lambda^{-r}\cdot \Delta(F,g)\}.
\end{equation}
\end{proposition}
Clearly this proposition, together with Lemma  \ref{lm:cpt} and Proposition \ref{pp:hyp},  completes the proof of Theorem \ref{th:localmain}.

\subsection{Beginning of the proof of Proposition \ref{pp:ctrnorm}}
Now we begin the proof of  Proposition \ref{pp:ctrnorm}.  
For the proof, it is enough to show that the central part $\widehat{\L}_{\ctr}$ of the lift $\widehat{\L}$ extends naturally to a bounded operator 
\[
\widehat{\L}_{\ctr}:L^2(\real^{2d}\oplus \real^{2d+1};\weight^r_\aniso)\to L^2(\real^{2d}\oplus \real^{2d+1};\weight^r_\aniso)
\]
and that the  operator norm of the extension is bounded by (\ref{eqn:cst}).
In other words, it is enough to prove
\begin{equation}\label{eqn:ctrest}
\|\mto \|_{L^2}\le C_0\cdot \max\{\Lambda(F,g),\; \|g\|_{\infty}\cdot \lambda^{-r}\cdot \Delta(F,g)\}
\end{equation}
for the operator
\[
\mto:\Sch(\real^{2d}\oplus \real^{2d+1})\to \Sch(\real^{2d}\oplus \real^{2d+1}),\quad 
\mto u= \weight_{\mathrm{aniso}}^r\cdot \widehat{\L}\left((\weight_{\mathrm{aniso}}^r)^{-1}\cdot X_{\ctr,0} \cdot u\right)
\]
where (and henceforth) we set 
\[
X_{\ctr,0}(z_\dag,\eta)=X_{\ctr}(z_\dag,\eta)\cdot (1-X_{0}(z_\dag,\eta))
\]
for simplicity. 
We write the  operator $\mto$ as an integral operator 
\begin{equation}
\mto u(x_\dag,\xi)=\int \frac
{\weight^r_{\mathrm{aniso}}(x_\dag,\xi)\cdot X_{\ctr,0}(z_\dag,\eta)}
{\weight^r_{\mathrm{aniso}}(z_\dag,\eta)}\cdot K(x_\dag,\xi;z_\dag,\eta) u(z_\dag,\eta) dz_\dag d\eta
\end{equation}
where $K(x_\dag,\xi;z_\dag,\eta)$ is the kernel of $\widehat{\L}$ given in (\ref{eqn:K}). Note that, if we perform the integration with respect to the first variable $y_0$
 in $y=(y_0, y_\dag)$ in (\ref{eqn:K}), we obtain the following expression of the kernel $K(x_\dag,\xi;z_\dag,\eta)$:
\begin{equation}\label{eqn:K2}
K(x_\dag,\xi;z_\dag,\eta)=\int \hat{g}(\xi_0-\eta_0, y_\dag)\cdot \Phi(x_\dag, \xi_0, z_\dag,\eta_0;y_\dag) 
\cdot e^{i\tau( x_\dag, \xi_\dag;z_\dag,\eta_\dag;y_\dag) } dy_\dag
\end{equation}
where
\begin{align*}
&\hat{g}(\xi_0, y_\dag)=(2\pi)^{-1/2}\cdot \int e^{-i\xi_0 y_0}\cdot g(y_0, y_\dag)\, dy_0
\intertext{and }
&\Phi(x_\dag, \xi_0;z_\dag,\eta_0;y_\dag) =\frac{\langle\xi_0\rangle^{d/2}\cdot \langle \eta_0\rangle^{d/2}}{\pi^d\cdot (2\pi)^{2d}}\cdot e^{-\langle\xi_0\rangle|x_\dag-y_\dag|^2/2-\langle\eta_0\rangle|F_\dag(y_\dag)-z_\dag|^2/2},\\
&\tau(x_\dag, \xi_\dag;z_\dag,\eta_\dag;y_\dag)=\xi_\dag\cdot ((x_\dag/2)-y_\dag) +\eta_\dag\cdot (F_\dag(y_\dag)-(z_\dag/2))+ \eta_0 \cdot  f(y_\dag).
\end{align*}

\subsection{Almost orthogonal decomposition of the operator $\mto$}
In this subsection, we decompose the operator $\mto$ into countably many operators, which are "almost orthogonal" to each other and then reduce the claim (\ref{eqn:ctrest}) to a similar claim for each of them. 

First we consider a simple partition of unity on the real line 
\[
\{q_k(t):\real\to [0,1]\}_{k\in \integer},\qquad q_k(t)=\chi(t-k+1)-\chi(t-k+2).
\]
Then, pulling back this partition of unity by the homeomorphism
\[
\gamma:\real\to \real,\quad \gamma(t)=\begin{cases}
\sqrt{t},&\mbox{ if $t\ge 0$;}\\
-\sqrt{-t},&\mbox{ if $t<0$,}
\end{cases}
\]
we define another partition of unity
\[
\{\tilde{q}_k:\real\to [0,1]\}_{k\in \integer},\qquad \tilde{q}_k(t)=q_k\circ \gamma(t).
\] 
The support of $\tilde{q}_k$ for $k> 0$ (resp. $k<0$) is contained in the interval 
\[
[(k-(2/3))^2, (k+(2/3))^2]\quad \mbox{iresp.  $[-(k+(2/3))^2, -(k-(2/3))^2]$).}
\]
Hence we can take a small constant $c>0$ such that 
\begin{equation}\label{eqn:dist0}
 d(\supp\, \tilde{q}_k, \supp\, \tilde{q}_{k'})\ge c\cdot \max\{ k,k'\} \quad \mbox{whenever $|k-k'|\ge 2$.}\quad
\end{equation}

We decompose the operator $\mto$ into countably many operators 
\[
\mto_k:\Sch(\real^{2d}\oplus \real^{2d+1})\to \Sch(\real^{2d}\oplus \real^{2d+1})\qquad \mbox{for $k\in \integer$}
\]
 defined by
\[
\mto_k u=\mto(\tilde{q}_k\cdot  u)
\]
where we identify $\tilde{q}_k$ with the function $(x_\dag,\xi)\mapsto \tilde{q}_k(\xi_0)$ on $\real^{2d}\oplus \real^{2d+1}$. Note that the operator $\mto_k$ with  $k^2\le N/2$ vanishes 
because so does  the term $X_{\ctr,0}(z,\eta)\cdot \tilde{q}_k(\eta_0)$ in its definition.

The operators $\mto_k$ are almost orthogonal to each other in the following sense.
\begin{lemma}\label{lm:ortho}
For arbitrarily large  $\nu>0$, there exists a constant $C_\nu>0$, which may depend on $F$ and $g$, such that, if $|k-k'|\ge 2$, we have
\[
|(\mto_k u, \mto_{k'}v)_{L^2}|\le C_\nu\cdot \max\{k,k'\}^{-\nu}\cdot \|u\|_{L^2}\cdot \|v\|_{L^2}.
\]
\end{lemma}
\begin{proof}
We estimate the kernel of $\mto_k^*\circ \mto_{k'}$ by using Lemma \ref{lm:ker} and (\ref{eqn:dist0}), and then apply the Schur test to obtain the conclusion. We omit the details as it is straightforward and tedious. 
\end{proof}

We next show that the required estimate (\ref{eqn:ctrest}) follows if we show 
\begin{equation}\label{eqn:ctrestk}
\|\mto_k\|_{L^2}\le C_0\cdot \max\{\Lambda(F,g),\; \|g\|_{\infty}\cdot \lambda^{-r}\cdot \Delta(F,g)\}
\end{equation}
for all $k$. Take a function $u\in \Sch(\real^{2d}\oplus \real^{2d+1})$ arbitrarily and set $u_k=\mathbf{1}_{\supp\, \tilde{q}_k}\cdot u$. Since the intersection multiplicity of $\supp\, \tilde{q}_k$ is bounded by $2$, we have
\begin{equation}\label{eqn:sumuk}
\sum_{k}\|u_k\|^2\le 2\|u\|^2_{L^2}.
\end{equation}
Let us write $\|\mto u\|^2_{L^2}$ as 
\begin{align*}
\|\mto u\|^2_{L^2}&= \sum_{k,k'} (\mto_k u_k, \mto_{k'}u_{k'})_{L^2}\\
&\le \sum_{|k-k'|\le 1} \|\mto_k u_k\|_{L^2}\cdot \|\mto_{k'} u_{k'}\|_{L^2}+\sum_{|k-k'|\ge 2} (\mto_k u_k, \mto_{k'}u_{k'})_{L^2}.
\end{align*}
From Lemma \ref{lm:ortho}, (\ref{eqn:sumuk}) and the fact that $\mto_k$ with  $k^2\le N/2$ vanishes, the second sum on the last line above  should be much smaller than $\|u\|_{L^2}^2$ in ratio, provided that the constant $N$ is sufficiently large. 
For the first sum, the estimate (\ref{eqn:ctrestk}), together with  (\ref{eqn:sumuk}), will yield the estimate
\begin{align*}
\sum_{|k-k'|\le 1} \|\mto_k u_k\|_{L^2} \|\mto_{k'} u_{k'}\|_{L^2}&\le
\sum_{|k-k'|\le 1}\frac{ \|\mto_k u_k\|_{L^2}^2+ \|\mto_{k'} u_{k'}\|_{L^2}^2}{2}\le 
 3 \sum_{k} \|\mto_k u_k\|^2_{L^2}\\
 & \le 3 \left(C_0\max\{\Lambda(F,g),\; \|g\|_{\infty}\lambda^{-r} \Delta(F,g)\}\right)^2 \sum_{k} \|u_k\|_{L^2}^2\\
 & \le 6 \left(C_0\max\{\Lambda(F,g),\; \|g\|_{\infty}\lambda^{-r}\Delta(F,g)\}\right)^2  \|u\|_{L^2}^2.
\end{align*}
Therefore the required estimate (\ref{eqn:ctrest}) follows from (\ref{eqn:ctrestk}).

Next we will go through a similar procedure as above, but this time we consider a partition of unity in the space variables. Below we take and fix an arbitrary $k\in \integer$ such that $k^2\ge N/2$. Take a partition of unity on $\real^{2d}$,
\[
Q_{k,\k}:\real^{2d}\to [0,1],\qquad  \k=(k_1,k_2,\cdots, k_{2d})\in \integer^{2d}
\]
 defined by 
\[
Q_{k,\k}(x_\dag)= \prod_{j=1}^{2d} q_{k_j}(k^{1-\delta}\cdot x_j)\quad \mbox{for $x_\dag=(x_1, x_2, \cdots, x_{2d})$}
\]
where $q_{k}(\cdot)$ is the function introduced in the beginning of this subsection and $\delta>0$ is a small number that will be specified later. 
Note that the supports of $Q_{k,\k}$ and $Q_{k,\k'}$ intersects only if $
\max_{i} |k_i-k'_i|\le 1$, and otherwise  we have
\[
d(\,\supp\, Q_{k,\k}\,, \,\supp\, Q_{k,\k'}\,)\ge c\cdot  |k|^{-1+\delta}\cdot \max_{1\le i\le 2d} |k_i-k'_i|
\]
for some small constant $c>0$ independent of $k$ and $\k$. 

By using the partition of unity $\{Q_{k,\k}\}_{\k\in \integer^{2d}}$, we decompose the operator $\mto_k$ into countably many operators 
\[
\mto_{k,\k} u= \mto_k (Q_{k,\k}\cdot u)=\mto (\tilde{q}_k\cdot Q_{k,\k}\cdot u), \qquad 
\mbox{for  $\k\in \integer^{2d}$,}
\]
where we identify $Q_{k,\k}$ with the function $(x,\xi)\mapsto Q_{k,\k}(x)$ on $\real^{2d}\oplus \real^{2d+1}$.

Arguing in the similar way as in the proof of Lemma \ref{lm:ortho}, we can show
\begin{lemma}\label{lm:ortho2}
For any $\nu>0$, there exists a constant $C_\nu>0$, which depends on $F$ and $g$, such that, if 
\[
d(\k, \k'):=\max_{1\le i\le 2d} |k_i-k'_i|\ge 2,
\]
it holds
\[
|(\mto_{k,\k} u, \mto_{k,\k'}v)_{L^2}|\le C_\nu\cdot |k|^{-\delta\nu}\cdot d(\k, \k')^{-\nu}\cdot \|u\|_{L^2}\cdot \|v\|_{L^2}.
\]
\end{lemma}

Then, proceeding similarly to  the argument in the paragraph succeeding to  Lemma~\ref{lm:ortho}, we see that the claim (\ref{eqn:ctrestk}) follows if we prove the estimate 
\begin{equation}\label{eqn:ctrestk2}
\|\mto_{k,\k}\|\le C_0 \cdot \max\{\Lambda(F,g), \|g\|_{\infty}\cdot \lambda^{-r}\cdot \Delta(F,g)\}\quad \mbox{for any $k\in \integer$ and $\k\in \integer^{2d}$.}
\end{equation}

In conclusion, we reduced the required estimate (\ref{eqn:ctrest}) on $\mto$ to the uniform  estimate (\ref{eqn:ctrestk2})  on the operators $\mto_{k,\k}$. 

\def\zero{\mathbf{0}}

\subsection{Approximation by linearization}
In this subsection, we prove the estimate (\ref{eqn:ctrestk2})  and finish the proof of Proposition \ref{pp:ctrnorm}. By changing the coordinates by elements of the transformation group $\Af$, we may and will assume that
\[
\k=\zero\quad\mbox{ and }\quad F(0)=0
\]
without loss of generality. 
Let $B:\real^{2d}\to \real^{2d}$ be the linearization of $F_\dag$ at the origin $0$, that is, we set $B=(DF_\dag)_{0}$. 
Note that it satisfies the conditions (B1)-(B3) in Subsection \ref{ss:l0}.

The operator $\mto_{k,\zero}$ is an integral operator 
\begin{equation}\label{l1}
\mto_{k,\zero}u(x_\dag,\xi)= \int \frac{\weight^r_\aniso(x_\dag,\xi)}{\weight^r_\aniso(z_\dag,\eta)} \cdot  K_{k}(x_\dag,\xi;z_\dag,\eta) \, u(z_\dag,\eta)\, dz_\dag d\eta
\end{equation}
where 
\begin{equation}\label{l2}
K_{k}(x_\dag,\xi;z_\dag,\eta)= \tilde{q}_k(\eta_0)\cdot Q_{k,\zero}(z_\dag) \cdot X_{\ctr,0}(z_\dag,\eta)\cdot K(x_\dag,\xi;z_\dag,\eta).
\end{equation}
and $K(x_\dag,\xi;z_\dag,\eta)$ is the kernel of $\widehat{\L}$ given  in (\ref{eqn:K2}) (or (\ref{eqn:K})).

As an approximation of the operator $\mto_{k,\zero}$, we introduce another  operator $\mto'_{k,\zero}$ that is defined by  (\ref{l1})  and (\ref{l2}) but with
\begin{itemize}
\setlength{\itemsep}{6pt plus 1pt minus 1pt}
\item  $\displaystyle
\frac{\weight^r_\aniso(x_\dag,\xi)}{\weight^r_\aniso(z_\dag,\eta)}$ in (\ref{l1}) replaced by $\displaystyle
\frac{\weight^r_\aniso(x_\dag,(k^2, \xi_\dag))}{\weight^r_\aniso(z_\dag,(k^2, \eta_\dag))}$
and,
\item  $K(x_\dag,\xi;z_\dag,\eta)$ in (\ref{l2}) replaced by 
\[
K'(x_\dag,\xi;z_\dag,\eta) =\int  \hat{g}(\xi_0-\eta_0,0)\cdot 
\Phi'_{k}(z_\dag;x_\dag;y_\dag)\cdot  e^{i\tau'_{k}(z_\dag,\eta_\dag;x_\dag,\xi_\dag;y_\dag)}dy_\dag
\]
\end{itemize}
where $\Phi'_{k}$ and $\tau'_{k}$ are defined respectively by
\begin{align*}
&\Phi'_{k}(x_\dag;z_\dag;y_\dag)=
\frac{k^{2d}}{\pi^{d}\cdot (2\pi)^{2d}}\cdot
  e^{-k^2 |x_\dag-y_\dag|^2/2- k^2|B(y_\dag)-z_\dag|^2/2}
\intertext{and}
&\tau'_{k}(x_\dag,\xi_\dag;z_\dag,\eta_\dag;y_\dag)= \xi_\dag\cdot ( (x_\dag/2)-y_\dag) + \eta_\dag\cdot ( B(y_\dag)-(z_\dag/2)).
\end{align*}
Compare the definition of the function $K'(\cdot)$ above and  that of  $K(\cdot)$ in (\ref{eqn:K2}).
To get the operator $\mto'_{k,\zero}$ from $\mto_{k,\zero}$, we  replaced
\begin{itemize}
\setlength{\itemsep}{3pt plus 1pt minus 1pt}
\item  the diffeomorphism $F$ by its linearization at the origin $0$,  
\item the function $\hat{g}(\xi_0, y_\dag)$ by  $\hat{g}(\xi_0, 0)$, 
\item $\xi_0$ and $\eta_0$ by $k^2$, 
\end{itemize}
and ignored the function $f$. 

Below we first show that the operator $\mto'_{k,\zero}$ satisfies the estimate (\ref{eqn:ctrestk2}) and then show that $\mto_{k,\zero}$ is well approximated by $\mto'_{k,\zero}$.

\begin{lemma}\label{lm:linear2}
There exists a constant $C_0>0$,  which does not depend on $F$ nor $g$, such that
\[
\|\mto'_{k,\zero}\|\le C_0\cdot \max\{\Lambda(F,g), \|g\|_{\infty}\cdot  \lambda^{-r}\cdot \Delta(F,g)\}\qquad \mbox{for any \;$k\in \integer$}.
\]
\end{lemma}
\begin{proof}  
In the proof below, we ignore the term $\tilde{q}_k(\eta_0)\cdot Q_{k, \zero}(z,\eta) \cdot X_{\ctr,0}(z,\eta)$ in (\ref{l2}), because the multiplication by such a function does not increase the $L^2$ norm of functions. Let us introduce three operators:
\begin{align*}
&G:L^2(\real)\to L^2(\real),\qquad\qquad \quad  \qquad \qquad G u(t)=\int \hat{g}(t-t',0)\cdot  u(t') dt',\\
&S_k:L^2(\real^{2d}\oplus \real^{2d})\to L^2(\real^{2d}\oplus \real^{2d}),\;\quad S_k  u(x_\dag,\xi_\dag)
=u(k^{-1}x_\dag, k\,\xi_\dag)\\
\intertext{and}
&W:\Sch(\real^{2d}\oplus \real^{2d})\to \Sch(\real^{2d}\oplus \real^{2d}),\quad W u(x_\dag,\xi_\dag)=\weight^r_{\aniso}(x_\dag,(k^2,\xi_\dag))\cdot u(x_\dag,\xi_\dag).
\end{align*}
We identify $L^2(\real^{2d}\oplus \real^{2d+1})$ with the tensor product 
$L^2(\real)\otimes L^2(\real^{2d}\oplus \real^{2d})$ by the natural extension of the correspondence 
\[
L^2(\real)\otimes L^2(\real^{2d}\oplus \real^{2d})\ni u \otimes v \longleftrightarrow
\varphi(x_\dag,\xi):=u(\xi_0)\cdot v(x_\dag,\xi_\dag)\in L^2(\real^{2d}\oplus \real^{2d+1})
\]
where $\xi=(\xi_0,\xi_\dag)$. Then the operator $\mto'_{k,\zero}$ is identified with the tensor product
\[
G\otimes (W\circ (S_k)^{-1}\circ \widehat{\L}_B\circ S_k\circ W^{-1}):L^2(\real)\otimes L^2(\real^{2d}\oplus \real^{2d})\to L^2(\real)\otimes L^2(\real^{2d}\oplus \real^{2d}).
\]
Since the operator $G$ is just the multiplication by $g$ viewed through the inverse Fourier transform, 
we have
\[
\|G:L^2(\real)\to L^2(\real)\|=\sup_{y_0\in \real} |g(y_0, 0)|.
\]
Since $S_k$ is a unitary operator, the operator norm of the  operator 
\[
W\circ S_k^{-1}\circ \widehat{\L}_B\circ S_k\circ W^{-1}:L^2(\real^{2d}\oplus \real^{2d})\to L^2(\real^{2d}\oplus \real^{2d})
\]
should  be same as that of 
\begin{equation}\label{lb}
\widehat{\L}_B:L^2(\real^{2d}\oplus \real^{2d}; W_{k^2})\to 
L^2(\real^{2d}\oplus \real^{2d}; W_{k^2})
\end{equation}
where $W_s:\real^{2d}\oplus \real^{2d}\to \real$ is the function defined in (\ref{wss}). As we discussed in Subsection \ref{ss:l0}, the operator norm of  (\ref{lb}) equals  that of
\[
\widehat{\L}_0\otimes \overline{\widehat{\L}_0}:L^2(\real^{2d}; \mathcal{V}_{k^2})\otimes L^{2}(\real^{2d})\to L^2(\real^{2d}; \mathcal{V}_{k^2})\otimes L^{2}(\real^{2d})
\]
and, from  Lemma \ref{lm:linFBI},  is bounded by $C_0\cdot \max\{ d(B)^{-1/2},  d(B)^{1/2}\cdot \lambda^{-r}\}$. 

Therefore, noting that $d(B)$ is proportional to $\det (DF_0|_{E^+})$, we conclude that  the operator norm of $\mto'_{k,\zero}$ with respect to the $L^2$ norm is bounded by 
\begin{align*}
C_0 \cdot \left(\sup_{y_0\in \real} |g(y_0, 0)|\right)\cdot 
&\max\{ d(B)^{-1/2}, d(B)^{1/2}\cdot \lambda^{-r}\}\\
&\qquad \qquad \le C_0 \max\{\Lambda(F,g), \|g\|_\infty\cdot \lambda^{-r}\cdot \Delta(F,g)\}.
\end{align*}
This completes the proof of Lemma \ref{lm:linear2}.
\end{proof}
The last step of our proof is the following approximation lemma. Recall that our construction depend on the constant $N>0$ and that both of $\LL_{k,\zero}$ and $\LL'_{k,\zero}$ vanish if $k^2<N/2$. 

\begin{lemma}\label{lm:diff2}
For any $\epsilon>0$, we may take the constant $N$ so large  that
\[
\|\LL_{k,\zero}-\LL'_{k,\zero}\|\le \epsilon\quad  \mbox{for all $k\in \integer$.}
\]
\end{lemma}
\begin{proof}
We write  $
\mathcal{K}(x_\dag,\xi;z_\dag,\eta)$ and $\mathcal{K}'(x_\dag,\xi;z_\dag,\eta)$ for  the kernels of the operators $\LL_{k,\zero}$ and $\LL'_{k,\zero}$ respectively and estimate the difference between them. 
We will suppose that the point $(z_\dag,\eta)$ satisfies 
\begin{equation}\label{eqn:xxi}
\|z_\dag\|\le 2\sqrt{2d}\cdot |k|^{-1+\delta},\qquad 
\|\eta_0-k^2\|\le |k|^{1+\delta},\qquad  \|\eta-\eta_0 \cdot\alpha_0 (z_\dag)\|\le |k|^{2\tau+\delta}
\end{equation}
because  both of $
\mathcal{K}(x_\dag,\xi;z_\dag,\eta)$ and $\mathcal{K}'(x_\dag,\xi;z_\dag,\eta)$ should vanish otherwise. 
(Recall that both of the kernels contain the term $\tilde{q}_k(\eta_0)\cdot Q_{k,\zero}(z_\dag) \cdot X_{\ctr,0}(z_\dag,\eta)$.)
Also, we may and do suppose  that  the point $(x_\dag,\xi)$ satisfies 
\begin{equation}\label{eqn:zzeta}
\|x_\dag\|\le |k|^{-1+2\delta},\quad \|\xi_0-k^2\|\le |k|^{1+2\delta},   \quad \|\xi-\xi_0\cdot\alpha_0(x_\dag)\|\le |k|^{2\tau+2\delta}.
\end{equation}
This is because, from the estimate in Lemma \ref{lm:ker}, the contributions of the parts of the kernels on outside of such region to the operator norms of $\mto_{k,\zero}$ and $\mto'_{k,\zero}$ are of order $\mathcal{O}(k^{-\infty})$ and hence we may assume it arbitrarily small by taking large constant~$N$. 
For the same reason, we may and do suppose that the integrations with respect to the variable $y_\dag$ in the definitions of  $K(x_\dag,\xi;z_\dag,\eta)$ and $K'(x_\dag,\xi;z_\dag,\eta)$ are restricted to the region
\begin{equation}\label{eqn:y}
\|x_\dag-y_\dag\|\le |k|^{-1+\delta},\quad \|F_\dag(y_\dag)-z_\dag\|\le |k|^{-1+\delta}.
\end{equation}

If we take sufficiently small constant $\delta>0$ according to the choice of $r$ and $\tau$, one can check that all of the following quantities are bounded by $|k|^{-2/3}$, provided  that (\ref{eqn:xxi}), (\ref{eqn:zzeta}) and (\ref{eqn:y}) hold and that $|k|$ is sufficiently large:
\begin{align*}
&| \eta_\dag\cdot F_\dag(y_\dag)- \eta_\dag\cdot  B(y_\dag)|,\quad |\xi_0-k^2|/k^2,\quad 
|\eta_0-k^2|/k^2,\\
&\left|k^2 |x_\dag-y_\dag|^2/2-\langle \xi_0 \rangle |x_\dag-y_\dag|^2/2\right|,\quad \left|k^2 |B(y_\dag)-z_\dag|^2/2-\langle \eta_0\rangle|F_\dag(y_\dag)-z_\dag|^2/2\right|,\\
&|\hat{g}(\xi_0-\eta_0,y_\dag)-\hat{g}(\xi_0-\eta_0,0)|/\langle \xi_0-\eta_0\rangle^{-2},\quad \mbox{and}\\
&|\weight^r_{\aniso}(x_\dag,(k^2, \xi_\dag))-\weight^r_{\aniso}(x_\dag,\xi)|/\weight^r_{\aniso}(x_\dag,\xi),\\
&|\weight^r_{\aniso}(z_\dag,(k^2, \eta_\dag))-\weight^r_{\aniso}(z_\dag,\eta)|/\weight^r_{\aniso}(z_\dag,\eta).
\end{align*}
Further, under the same assumptions, we have  that 
\[
 |\langle \eta_0, f(y_\dag)\rangle | \le |k|^{-2/3}
\] 
from  Lemma \ref{pp:basic}, and also that 
\[
 C^{-1}\cdot  |k|^{-r((2\tau-1)+\delta)}\le \weight^r_{\aniso}(z_\dag, (k^2, \eta_\dag)) \le C\cdot  |k|^{2r((2\tau-1)+\delta)}
\]
and
\[
 C^{-1}\cdot  |k|^{-r((2\tau-1)+2\delta)}\le \weight^r_{\aniso}(x_\dag, (k^2, \xi_\dag))\le C\cdot  |k|^{2r((2\tau-1)+2\delta)}
\]
for a large constant $C>0$.
Therefore, comparing $\mathcal{K}(x_\dag,\xi;z_\dag,\eta)$ and $\mathcal{K}'(x_\dag,\xi;z_\dag,\eta)$ 
and using the estimates above, we obtain
\[
|\mathcal{K}(x_\dag,\xi;z_\dag,\eta)- \mathcal{K}'(x_\dag,\xi;z_\dag,\eta)|\le C\cdot |k|^{-1/2} \cdot 
\langle \xi_0-\eta_0\rangle^{-2}\cdot \int \Phi'_{k}(z_\dag;x_\dag;y_\dag)\, dy_\dag.
\]
Recalling the restrictions (\ref{eqn:xxi}) and (\ref{eqn:zzeta}) on the range of $(z,\eta)$ and $(x,\xi)$ and employing the Schur test, 
we see that this implies
\[
\|\LL_{k,\zero}-\LL'_{k,\zero}\|_{L^2}
\le C\cdot  |k|^{-1/4}
\]
provided that $\delta>0$ is sufficiently small. 
Therefore, for any $\epsilon>0$, we have 
\[
\left\|\LL_{k,\zero}-\LL'_{k,\zero}\right\|_{L^2}<\epsilon
\]
if $|k|$ is sufficiently large,  and otherwise  we may assume that both of $\LL_{k,\zero}$ and $\LL'_{k,\zero}$ vanish by letting the constant $N$ be large. 
\end{proof}

\section{The lower estimate}\label{sec:lb}
In this section, we compete the proof of the main theorem by proving 
\[
\rho_{\mathrm{ess}}(\L^t:H_\aniso^r(M)\to H_\aniso^r(M))\ge \Lambda^t.
\]
For this purpose, it is enough to show that there exist a small constant $c>0$, which does not depend on $t$,  and an infinite dimensional subspace $H(t)\subset H^r_\aniso(M)$ for  sufficiently large $t$ such that 
\begin{equation}\label{eqn:ht}
\| \L^t(u)\|^r_\aniso\ge c\cdot \left(\sup_{x\in M} \frac{|g^t(x)|}{\sqrt{\det (DF^t|_{E^u})(x)}}\right)\cdot \|u\|^r_{\aniso}
\quad \mbox{for all $u\in H(t)$. }
\end{equation}
We can construct such subspace $H(t)$ as follows\footnote{The following argument may be a bit rough but should be easy to put into a rigorous argument once we went through the previous sections.}. 
First take a point $x_*=x_*(t)$ that attains the supremum in (\ref{eqn:ht}) and   
choose a coordinate chart $\kappa_a:U_a\to V_a$ so that $F^t(x_*)\in U_a$.
Take large integer $m>0$ and a sparse increasing sequence $\{n_k\}_{k=1}^\infty$ of integers and then define a sequence of functions on $\real^{2d+1}$ by
\[
\widetilde{\varphi}_k(y)=c_k\cdot \Phi_{x_\dag,\xi_k}(y)\cdot \chi(m |y-x|)
\]
where $\Phi_{x_\dag,\xi_k}(\cdot)$ is the function defined in Subsection \ref{ss:pFBI}, 
\[
x=(x_0,x_\dag):=\kappa_a(F^t(x_*)),\qquad \xi_k=n_k\cdot \alpha_0(x)
\]
 and $c_k$ is a normalization constant such that $\|\widetilde{\varphi}_k\|_{L^2}=1$.
If we take large $m$ and  sufficiently sparse sequence $\{n_k\}$, one can show the following properties: 
\begin{itemize}
\setlength{\itemsep}{6pt plus 1pt minus 1pt}

\item $\varphi_k:=\widetilde{\varphi}_k\circ \kappa_a$, $k\ge 1$, are almost orthogonal to each other in $H^r_\aniso(M)$, 
\item $\L^t(\varphi_k)$, $k\ge 1$, are also almost orthogonal to each other in $H^r_\aniso(M)$, and 
\item the inequality (\ref{eqn:ht}) holds with $u=\varphi_k$ for $k\ge 1$. 
\end{itemize}
In fact, if $n_k$ are $n_{k'}$ are apart from each other, so are the frequencies  of $\varphi_k$ and $\varphi_{k'}$ (resp. $\L^t(\varphi_k)$ and $\L^t(\varphi_{k'})$)  in the flow direction and, therefore,  they are almost orthogonal to each other in $H^r_\aniso(M)$. 
To check the third claim,  note that  the function $\varphi_k$ is localized in a small neighborhood of $F^t(x_*)$ on which we may suppose that $F^t$ viewed in the local coordinate is  almost linear and $g^t$ is almost constant. 
 (Notice that we take $m$ and $\{n_k\}$ according to~$t$.) 
If~$F^t$ were linear and $g$ were constant, one could obtain the inequality (\ref{eqn:ht}) for $u=\varphi_k$ by a straightforward estimate using the argument in Section \ref{s:FBI}, \ref{s:pFBI} and Subsection \ref{ss:l0}. 
To conclude, we employ an approximation argument similar to (but much simpler than) that in the proof of Lemma~\ref{lm:diff2}. 
(We ask the readers to work a bit to check the details.) 
The infinite dimensional subspace $H(t)$ spanned by $\{\varphi_k\}_{k=1}^\infty$ satisfies the required property.

\appendix

\section{Proof of Lemma \ref{lm:sob} and Corollary \ref{cor:sob}}
We consider the composition $\pFBI\circ \Fourier^{-1}:L^2(\real^{2d})\to L^2(\real^{2d}\oplus \real^{2d+1})$. The Schwartz kernel of this operator is 
\[
K(x_\dag,\xi;\eta)=\int e^{i\eta y} \cdot \overline{\Phi_{x_\dag,\xi}(y)} dy
\]
Calculating a Gaussian integral, we find
\[
K(x_\dag,\xi;\eta)=\langle \xi_0\rangle^{-d/2}\cdot \pi^{-d/2}\cdot e^{-\langle \xi_0\rangle^{-1}|\eta_\dag-\xi_\dag|^2/2+i(\eta_\dag-\xi_\dag/2)x_\dag}\cdot \delta(\xi_0-\eta_0)
\]
For given function $u\in \Sch(\real^{2d+1})$, we set $
\hat{u}=\Fourier u$ and $\check{u}=\pFBI u$. 
Then it holds
\[
(\|u\|'_{H^r})^2=\|w^{r}\cdot \pFBI\circ \Fourier^{-1} \hat{u}\|^2_{L^2}
\]
where $w^r(x,\xi)=\langle \|\xi\|\rangle^r$, and the right hand side can be written as
\[
\|w^r\cdot \pFBI\circ \Fourier^{-1} \hat{u}\|_{L^2}=\int \hat{u}(\eta)\cdot  \overline{\hat{u}(\eta')}\cdot  K(x_\dag,\xi;\eta)\cdot \overline{K(x_\dag,\xi;\eta')}\cdot \langle \|\xi\|\rangle^{2r}\, dx_\dag d\xi d\eta d\eta'
\]
By calculation, we see that 
\[
\int  K(x_\dag,\xi;\eta)\cdot \overline{K(x_\dag,\xi;\eta')}\cdot \langle \xi\rangle^{2r}\, dx_\dag  d\xi =
\delta(\eta-\eta')\cdot \int  \langle \xi\rangle^{2r} \cdot \frac{e^{-\langle \eta_0\rangle^{-1}|\eta_\dag-\xi_\dag|^2}}{\langle \eta_0\rangle^d\cdot \pi^{d}} d\xi_\dag
\]
and that, for some constant $C>0$, 
\[
\int  \langle \xi\rangle^{2r} \cdot \frac{e^{-\langle \eta_0\rangle^{-1}|\eta_\dag-\xi_\dag|^2}}{\langle \eta_0\rangle^d\cdot \pi^{d}} d\xi_\dag \le C\langle \|\eta\|\rangle^{2r}.
\]
Therefore we have
\[
(\|u\|'_{H^r})^2=\|w^{r}\cdot \pFBI\circ \Fourier^{-1} \hat{u}\|_{L^2}^2\le C\|w^r\cdot \hat{u}\|^2_{L^2}
=C\|u\|_{H^r}^2
\]
Next we show the estimate in the opposite direction. Note that we may write $\|u\|_{H^r}^2=\| w^r\cdot \Fourier\circ \pFBI^* \check{u}\|_{L^2}^2$ as
\[
\|w^r\cdot \Fourier\circ \pFBI^* \check{u}\|_{L^2}=
\int \check{u}(x,\xi)\cdot  \overline{\check{u}(x'_\dag,\xi')}\cdot  K(x_\dag,\xi;\eta)\cdot \overline{K(x'_\dag,\xi';\eta)}\cdot \langle \|\eta\|\rangle^{2r}\, dx_\dag d\xi d\eta d\eta'
\]
Since 
\begin{align*}
&\left|\int  K(x_\dag,\xi;\eta)\cdot \overline{K(x'_\dag,\xi';\eta)}\cdot \langle \|\eta\|\rangle^{2r}\, dx_\dag d\xi\right| \\
&\qquad \qquad =\frac{\delta(\xi_0-\xi'_0)}
{\langle \xi_0\rangle^{d} \pi^d}
 \int \delta(\xi_0-\eta_0) \cdot \langle \|\eta\|\rangle^{2r}\cdot  e^{-\langle \xi_0\rangle^{-1}|\eta_\dag-\xi_\dag|^2-\langle \xi_0\rangle^{-1}|\eta_\dag-\xi'_\dag|^2} d\eta
\end{align*}
and since 
\[
\frac{1}
{\langle \xi_0\rangle^{d} \cdot \pi^d}
 \int \delta(\xi_0-\eta_0) \cdot \langle \|\eta\|\rangle^{2r}\cdot  e^{-\langle \xi_0\rangle^{-1}|\eta_\dag-\xi_\dag|^2-\langle \xi_0\rangle^{-1}|\eta_\dag-\xi'_\dag|^2} d\eta
 \le C\langle \|\xi\|\rangle^r
\]
for some constant $C>0$, it holds
\[
\|u\|_{H^r}^2=\|w^r\cdot \Fourier\circ \pFBI^* \check{u}\|_{L^2}^2\le C\|w^r\cdot \check{u}\|^2_{L^2}
=C(\|u\|'_{H^r})^2.
\]
This finishes the proof of Lemma \ref{lm:sob}.

Corollary \ref{cor:sob} is essentially  a consequence of   Lemma \ref{lm:sob} and the fact that there exists a constant $C=C(K)>0$ for each compact subset 
$K\Subset \real^{2d+1}$ such that 
\[
C^{-1}\langle \|\xi\|\rangle^{-2r} \le  \weight^r_\aniso(x_\dag,\xi) \le C\langle \|\xi\| \rangle^{2r}\qquad \mbox{for all $x_\dag \in K$ and $\xi \in \real^{2d+1}$.}
\]
Actually, even if the support of $u$ is contained in a compact subset $K$, the support of $\pFBI u$ will not be contained in $K\times \real^{2d+1}$. But a tedious (and rather standard) argument using the fact that $\pFBI u$ decay rapidly outside the subset $K\times \real^{2d+1}$ give the required estimate. 

\bibliographystyle{amsplain}
\bibliography{mybib}

\end{document}